\documentclass[11pt]{amsart}
\usepackage[dvips]{graphicx} 
\usepackage{amssymb,amscd,color,latexsym,epsfig,xypic,hyperref,soul}
\usepackage[fleqn]{nccmath}
\usepackage[mathscr]{euscript}
\DeclareFontFamily{OT1}{rsfs}{}
\DeclareFontShape{OT1}{rsfs}{n}{it}{<-> rsfs10}{}
\DeclareMathAlphabet{\curly}{OT1}{rsfs}{n}{it}

\DeclareFontFamily{U}{mathb}{\hyphenchar\font45}
\DeclareFontShape{U}{mathb}{m}{n}{
      <5> <6> <7> <8> <9> <10> gen * mathb
      <10.95> mathb10 <12> <14.4> <17.28> <20.74> <24.88> mathb12
      }{}
\DeclareSymbolFont{mathb}{U}{mathb}{m}{n}

\makeatletter
\newcommand{\eqnum}{\refstepcounter{equation}\textup{\tagform@{\theequation}}}
\makeatother

\renewcommand\;{\hspace{.6pt}}

\newcommand\C{\mathbb C}

\newcommand\R{\mathbb R}

\newcommand\Z{\mathbb Z}

\newcommand\HH{\mathbb H}

\newcommand\PP{\mathbb P}

\newcommand\cO{\mathcal O}

\renewcommand\cH{\mathcal H}

\renewcommand\t{\mathfrak t}

\renewcommand\({\big(}
\renewcommand\){\big)}

\renewcommand\]{\big]}
\newcommand\wt{\widetilde}

\makeatletter
\renewcommand{\so}{\ \ext@arrow 0359\Rightarrowfill@{}{\hspace{3mm}}\ }
\makeatother
\newcommand{\rt}[1]{\xrightarrow{\ #1\ }}
\newcommand\To{\longrightarrow}

\newcommand\into{\hookrightarrow}
\newcommand\INTO{\ \ar@{^(->}[r]<-.2ex>}
\newcommand{\Into}{\,\ensuremath{\lhook\joinrel\relbar\joinrel\rightarrow}\,}

\newcommand\Mapsto{\ \longmapsto\ }
\newcommand\Mapsfrom{\ \mathrel{\reflectbox{\ensuremath{\longmapsto}}}\ }
\newcommand\ip{\lrcorner}

\renewcommand\_{^{}_}

\newcommand\take{\!\smallsetminus\!}
\newcommand{\mat}[4]{\left(\begin{array}{cc} \!\!#1 & #2\!\! \\ \!\!#3 &
#4\!\!\end{array}\right)}
\newcommand\<{\langle}
\renewcommand\>{\rangle}
\newcommand\Langle{\big\langle}
\newcommand\Rangle{\big\rangle}
\newfont{\bigtimesfont}{cmsy10 scaled \magstep5}
\newcommand{\bigtimes}{\mathop{\lower0.9ex\hbox{\bigtimesfont\symbol2}}}
\renewcommand\={\ =\ }

\newcommand\udot{^{\bullet}}

\DeclareMathSymbol{\lefttorightarrow}{3}{mathb}{"FC}
\DeclareMathSymbol{\righttoleftarrow}{3}{mathb}{"FD}

\newcommand\vir{\operatorname{vir}}
\newcommand\vd{\operatorname{vd}}

\newcommand\coker{\operatorname{coker}}
\newcommand\im{\operatorname{im}}
\newcommand\id{\operatorname{id}}

\renewcommand\Im{\operatorname{Im}}
\newcommand\Hom{\operatorname{Hom}}

\newcommand\End{\operatorname{End}}

\newcommand\Bl{\operatorname{Bl}}

\renewcommand\Re{\operatorname{Re}}

\renewcommand\t{\mathfrak t}

\newcommand\beq[1]{\begin{equation}\label{#1}}
\newcommand\eeq{\end{equation}}
\newcommand\beqa{\begin{eqnarray*}}
\newcommand\eeqa{\end{eqnarray*}}

\newcommand\arXiv[1]{\href{http://arxiv.org/abs/#1}{arXiv:#1}}
\newcommand\mathAG[1]{\href{http://arxiv.org/abs/math/#1}{math.AG/#1}}

\DeclareRobustCommand{\SkipTocEntry}[3]{}
\makeatletter
\newcommand\@dotsep{4.5}
\def\@tocline#1#2#3#4#5#6#7{\relax
  \ifnum #1>\c@tocdepth 
  \else
    \par \addpenalty\@secpenalty\addvspace{#2}%
    \begingroup \hyphenpenalty\@M
    \@ifempty{#4}{%
      \@tempdima\csname r@tocindent\number#1\endcsname\relax
    }{%
      \@tempdima#4\relax
    }%
    \parindent\z@ \leftskip#3\relax \advance\leftskip\@tempdima\relax
    \rightskip\@pnumwidth plus1em \parfillskip-\@pnumwidth
    #5\leavevmode #6\relax
    \leaders\hbox{$\m@th
      \mkern \@dotsep mu\hbox{.}\mkern \@dotsep mu$}\hfill
    \hbox to\@pnumwidth{\@tocpagenum{#7}}\par
    \nobreak
    \endgroup
  \fi}
\makeatother

\makeatletter \@addtoreset{equation}{section} \makeatother
\renewcommand{\theequation}{\thesection.\arabic{equation}}

\newtheorem*{thm*}{Theorem}
\newtheorem{lem}[equation]{Lemma}
\newtheorem{cor}[equation]{Corollary}
\newtheorem{prop}[equation]{Proposition}

\definecolor{darkgreen}{rgb}{0,.5,0}

\theoremstyle{definition}
\newtheorem{ex}[equation]{Example}
\newtheorem{rmk}[equation]{Remark}

\title[A Hopf index for isotropic sections]
{A Hopf index for isotropic sections of orthogonal bundles\vspace{-2mm}}
\author[M. Kool, J. Oh, J. V. Rennemo and R. P. Thomas]{Martijn Kool, Jeongseok Oh, J{\o}rgen Vold Rennemo and Richard Thomas}

\begin{document}
\begin{abstract} \noindent The Hopf index equates the multiplicity of a zero of a section of a vector bundle with a winding number. We give eight analogues for isotropic sections of bundles with quadratic form. There are applications to cosection localised 
virtual cycles and to DT$^4$ virtual cycles.\vspace{-2mm}\end{abstract}
\maketitle


Let $(E,q,o)$ be an $SO(2n,\C)$ bundle over a smooth manifold $Y$, where $q$ is a nondegenerate complex quadratic form and $o\in\Gamma\(\bigwedge^{2n}E\)$ is a complex orientation in the sense of \cite[Definition 2.1]{OT1}.

Suppose $s\in\Gamma(E)$ is \emph{isotropic}: $q(s,s)=0$. When its zero locus $Z(s)$ has the expected complex codimension $n$ there is a way to assign a multiplicity or cycle class $\sqrt e\;(E,s)$ to it \cite{EG,BJ,OT1}. When $(Y,E,s)$ are holomorphic and $Z(s)$ is smooth and connected this class is the fundamental class\footnote{This follows from \cite[Lemma 3.5]{OT1} applied to the maximal isotropic subbundle $\Lambda:=\im\(ds\colon T_Y|_{Z(s)}\to E|_{Z(s)}\)$ of $E|_{Z(s)}$, combined with \eqref{H3} below.}
$$
\sqrt e\;(E,s)\=\pm[Z(s)]
$$
up to sign, but in general is multiplied by a multiplicity. We give a number of formulae \eqref{OH1} to \eqref{OH8} for this multiplicity, and find it has some surprising features. The simplest case is when $Y=\C^2$ over which $E=\underline\C^4$ and $s$ has zeros only at the origin $0\in\C^2$. Restricting $s$ to the link $S^3\subset\C^2$ of $0\in\C^2$ defines a map to the quadric $Q=\PP^1\times\PP^1$ of isotropic lines in $\C^4$,
$$
\PP\(s|_{S^3}\)\ \colon\ S^3\ \To\ Q\ \cong\ \PP^1\times\PP^1\ \subset\ \PP(\C^4).
$$
Writing its homotopy class as
$$
\big[\PP\(s|_{S^3}\)\big]\ \in\ \pi_3(\PP^1\times\PP^1)\ \cong\ \Z\oplus\Z\ \ni\ (d_1,d_2),
$$
the formula \eqref{OH25} below for the multiplicity assigned to $Z(s)$ is
$$
\sqrt e\;(E,s)\=d_1-d_2.
$$
\emph{This can be zero}! For instance consider the isotropic section
\beq{eg}
s\=(x^2,y^2,xy,-xy)\ \text{ of }\ \underline\C^4\ \text{ over }\ Y\,=\,\C^2_{x,y}
\eeq
(with quadratic form $q=XY+ZW$ in the obvious notation).
We will calculate the bidegree to be $(d_1,d_2)=(1,1)$ so the multiplicity \emph{vanishes},
$$
\sqrt e\;(E,s)\=1-1\=0\ \ne\ 3\=\operatorname{length}Z(s).
$$
In this paper we will explain this phenomenon and give a number of ways to calculate $\sqrt e\;(E,s)$ via various ``orthogonal Hopf index theorems". We begin by listing eight versions of the classical Hopf index theorem before stating orthogonal analogues.

\section{Classical Hopf index}
The setting for the Hopf index \cite{H} is the following local data and notation.
\begin{itemize}
\item A rank $2n$ oriented real vector space $E$ with unit sphere $S(E)\cong S^{2n-1}$,
\item the trivial $SL(2n,\R)$ bundle $\underline E=E\times\R^{2n}$ over $\R^{2n}$,
\item a section $s\in\Gamma(\underline E)$ with zero locus $Z(s)=\{0\}\subset\R^{2n}$.
\end{itemize}
The standard orientation on $\R^{2n}$ defines a local intersection number of the graph $\Gamma_{\!s}\subset\underline E$ and the zero section $0_{\underline E}\subset\underline E$\,---\,a multiplicity on $Z(s)=\{0\}$ which we denote
\begin{fleqn}\begin{equation*}
\textbf{Hopf index}\hspace{3cm} 
e(\underline E,s)\ :=\ \Gamma_{\!s}\cdot0_{\underline E}\ \in\ \Z.
\end{equation*}\end{fleqn}
It is a localisation of the Euler class $e(\underline E)$ to $Z(s)$, a deformation invariant count of the number of zeros of $s$. We list some of the many ways to calculate it.\medskip

\noindent Restricting $s$ to the link $S^{2n-1}\subset\R^{2n}$ of $Z(s)=\{0\}$ and projecting it to $\underline E|\_0=E$,\vspace{-1mm}
\begin{fleqn}\beq{H1}\tag{{\bf H1}}
\textbf{Winding number}\qquad e(\underline E,s)\,=\,\deg s\,:=\,\big[s|_{S^{2n-1}}\big]\,\in\,\pi_{2n-1}(S(E))\,=\,\Z.
\eeq\end{fleqn}
Fixing an oriented complex structure on $E$ with space of complex lines $\PP(E)\cong\PP^{n-1}$,\vspace{-1mm}
\begin{fleqn}\beq{H2}\tag{{\bf H2}}
\textbf{Projectivisation}\qquad 
e(\underline E,s)\,=\,\big[\PP\(s|_{S^{2n-1}}\)\big]\,\in\,\pi_{2n-1}(\PP(E))\,\cong\,\Z.
\eeq\end{fleqn}
Writing $\R^{2n}$ as $\C^n$ and supposing $s$ is holomorphic from now on,\vspace{-1mm}
\begin{fleqn}\beq{H3}\tag{{\bf H3}}
\textbf{Holomorphic case}\qquad 
e(\underline E,s)\,=\,\operatorname{length}\,Z(s)\ \ge\ 0.
\eeq\end{fleqn}
Choosing holomorphic $s_t$ with simple zeros close to $0\in\C^n$ for $t\ne0$ and $s_0=s$,\vspace{-1mm}
\begin{fleqn}\beq{H4}\tag{{\bf H4}}
\textbf{Deforming to simple zeros}\qquad 
e(\underline E,s)\,=\,\#\big\{p\in\C^n\,\colon\,s_t(p)=0\big\}\,\ge\,0.
\eeq\end{fleqn}
If $s$ is homogeneous then $s|_{S^{2n-1}}$ descends to $\underline{s}\colon\PP^{n-1}\to\PP(E)$ and, if $n > 1$,\vspace{-1mm}
\begin{fleqn}\beq{H5}\tag{{\bf H5}}
\textbf{Homogeneous case}\qquad e(\underline E,s)\,=\,
(\deg\,\underline s)^{\frac n{n-1}}\,=\bigg(\int_{\PP^{n-1}}\underline s^*[\mathrm{pt}]\bigg)^{\frac n{n-1}}.
\eeq\end{fleqn}
For general holomorphic $s$ we instead use the fundamental class of the projectivised normal cone\footnote{The normal cone $C_{Z(s)/\C^n}\subset E|_{Z(s)}$ is the flat schematic limit $\lim_{t\to\infty}\Gamma_{\!ts}\subset E$. Then \eqref{H6} says that $[\PP(C_{Z(s)/\C^n})]$ is the cycle $[\PP(E_0)]=[\PP^{n-1}]$ with multiplicity $e(\underline E,s)$.} $[\PP(C_{Z(s)/\C^n})]\in H_{2n-2}\(\PP(E)|_{Z(s)}\)\cong H_{2n-2}\(\PP^{n-1}\)$,\vspace{-1mm}
\begin{fleqn}\beq{H6}\tag{{\bf H6}}
\textbf{Projective normal cone}\quad e(\underline E,s)\,=\,
\big[\PP\(C_{Z(s)/\C^n}\)\big]\in H_{2n-2}\(\PP^{n-1}\)\,\cong\,\Z.
\eeq\end{fleqn}
Since codim$\(Z(s)\subset\C^n\)=n$ the Koszul complex $\(\Lambda\udot \underline E^*,s\)$ resolves $\cO_{Z(s)}$ and\vspace{-1mm}
\begin{fleqn}\beq{H7}\tag{{\bf H7}}
\textbf{$K$-theoretic}\qquad e(\underline E,s)\,=\,\big[\(\Lambda\udot \underline E^*,s\)\big]\,\in\,K\(\mathrm{Coh}\;\{0\}\)\,=\,\Z.
\eeq\end{fleqn}
If $(\underline E,s)$ is equivariant with respect to a $T=\C^*$ action on $\C^n$ with $(\C^n)^T=\{0\}$,\vspace{-1mm}
\begin{fleqn}\beq{H8}\tag{{\bf H8}}
\textbf{Torus localisation}\qquad e(\underline E,s)\,=\,
e^T\!\(\underline E|\_0\)\big/e^T\(N_{0/\C^n}\).
\eeq\end{fleqn}
\vspace{-5mm}

\section{Orthogonal Hopf index}\label{OHI}
To state the orthogonal analogues of these results we need some data and notation.
\begin{itemize}
\item $(E,q,o)$ is a rank $2n\ge4$ complex vector space with nondegenerate quadratic form $q$ and compatible orientation $o\in\Lambda^{2n}E$ \cite[Definition 2.1]{OT1}.
\item $E_{\R}\subset E$ is a maximal positive definite subspace, so $E=E_{\R}\oplus iE_{\R}=E_{\R}\otimes\_{\R}\C$.
\item $\Lambda_\pm\subset E$ are maximal isotropic of \emph{sign} $(-1)^{|\Lambda_\pm|}=\pm1$ \cite[Definition 2.2]{OT1}.
\item $\underline E=E\times\C^n$ is the trivial $SO(2n,\C)$ bundle over $\C^n$.
\item $s\in\Gamma(\underline E)$ is \emph{isotropic} $q(s,s)=0$ and nonzero on $\C^n\take\{0\}$.
\end{itemize}
Then \cite[Section 3.2]{OT1} uses cosection localisation \cite{KL} to  assign a multiplicity\footnote{This is only strictly correct for $s$ holomorphic. For $s\in C^\infty(E)$ we take \eqref{OH1} as the definition, following \cite{BJ}. The definitions of \cite{BJ, OT1} are shown in \cite{OT2} to coincide when $s$ is holomorphic.}
\begin{fleqn}\begin{equation*}
\textbf{Orthogonal Hopf index}\hspace{6mm} \sqrt e(\underline E,s)\,\in\,\Z \quad\text{to the point }\{0\}\,=\,Z(s)\,\subset\,\C^n.
\end{equation*}\end{fleqn}\vskip 1mm

\begin{thm*} $\sqrt e(\underline E,s)$ is equal to \eqref{OH1}--\eqref{OH8} below.
\end{thm*}\vskip 1mm

\noindent More precise statements, fuller explanations and proofs follow later in the paper. \medskip

\noindent Writing $s=(s_+,s_-)$ with respect to $\underline E=\underline E_\R\oplus i\underline E_\R$,
\begin{fleqn}\beq{OH1}\tag{{\bf OH1}}
\textbf{Winding number}\ \,
\sqrt e\;(\underline E,s)=\deg s_+:=\big[s_+|_{S^{2n-1}}\big]\in\pi_{2n-1}(S(E_\R))=\Z,
\eeq\end{fleqn}
\begin{flalign}\label{OH15}\tag{{\bf OH1}$\mathbf{\tfrac12}$}
&& \sqrt e\;(\underline E,s)=\deg s_-:=\big[s_-|_{S^{2n-1}}\big]\in\pi_{2n-1}(S(iE_\R))=\Z.
\end{flalign}
Projecting $s$ to the quadric of isotropic lines $Q\subset\PP(E)=\PP^{2n-1}$ defines $\big[\PP\(s|_{S^{2n-1}}\)\big]$ in $\pi_{2n-1}(Q)\cong\Z\oplus (\Z/2\Z)$ (or $\Z\oplus\Z$ if $n=2$). Writing it as $(s_1,s_2)$,
\begin{fleqn}\beq{OH2}\tag{\textbf{OH2}}
\textbf{Projectivisation} \hspace{2cm}\sqrt e\;(\underline E,s)\,=\,s_1.
\eeq\end{fleqn}
Writing it (non-uniquely) in terms of different generators $[S(\Lambda_\pm)]$ of $\pi_{2n-1}(Q)$,
\beq{OH25}\tag{{\bf OH2}$\mathbf{\tfrac12}$}
\text{if}\ \big[\PP\(s|_{S^{2n-1}}\)\big]\,=\,d_1\big[S(\Lambda_+)\big]+d_2\big[S(\Lambda_-)\big]\,\text{ then }\,\sqrt e(\underline E,s)\,=\,d_1-d_2.
\eeq
If $s$ is holomorphic and factors through a maximal isotropic subbundle $\underline\Lambda\subset\underline E$,\vspace{-1mm}
\begin{fleqn}\beq{OH3}\tag{\textbf{OH3}}
\textbf{Holomorphic}\qquad \sqrt e\;(\underline E,s)\=(-1)^{|\underline\Lambda|\;}\mathrm{length}\,Z(s). 
\eeq\end{fleqn}
More generally this fails, as in the example \eqref{eg} with $n=2$ and $s=(x^2,y^2,xy,-xy)$,\vspace{-1mm}
\begin{fleqn}\begin{equation}\tag{{\;\bf OH3\hspace{-9.5mm}-----\hspace{-1mm}---\!\;\;}}\label{OH3'}
\textbf{Holomorphic}\mathbf{'}\qquad \sqrt e\;(\underline E,s)\,=\,0\,\ne\,\pm\;3\,=\,\pm\operatorname{length}Z(s)\ \text{ for }\eqref{eg}.
\eeq\end{fleqn}
If $s_t$ is holomorphic, isotropic with $s_0=s$ and simple zeros near $0\in\C^n$ for $t\ne0$,\vspace{-1mm}
\begin{fleqn}\beq{OH4}\tag{{\bf OH4}}
\textbf{Deforming to simple zeros}\quad \sqrt e\;(\underline E,s)\,=\!\sum_{p\in\C^n\colon s_t(p)=0}(-1)^{|\Lambda_p|},
\eeq\end{fleqn}
where $\Lambda_p\subset\underline E|_p$ is the maximal isotropic subspace $\im \(ds_t\colon T_p\;\C^n\to\underline E|_p\)$.\smallskip

\noindent For $s$ homogeneous $s|_{S^{2n-1}}$ descends to $\underline{s}\colon\PP^{n-1}\to Q$. Setting $d_\pm:=\int_{\PP^{n-1}}\underline s^*\big[\PP(\Lambda_\pm)\big]$,\vspace{-1mm}
\begin{fleqn}\beq{OH5}\tag{{\bf OH5}}
\textbf{Homogeneous case}\qquad \sqrt e\;(\underline E,s)\,=(-1)^{n-1}(d_++d_-)^{\frac1{n-1}}(d_+-d_-).
\eeq\end{fleqn}
For more general $s$ we instead use the (isotropic!) projective normal cone $\PP(C_{Z(s)/\C^n})$ $\subset\PP\(\underline E|_{Z(s)}\)$, supported on $Q\subset\PP(E)$ with class $\big[\PP(C_{Z(s)/\C^n})\big]\in H_{2n-2}(Q)$.\vspace{-1mm}
\begin{fleqn}\beq{OH6}\tag{{\bf OH6}}
\textbf{Normal cone}\qquad
\sqrt e\;(E,s)\=(-1)^{n-1}\int_{[\PP(C_{Z(s)/\C^n})]}\big[\PP(\Lambda_+)\big]-\big[\PP(\Lambda_-)\big].
\eeq\end{fleqn}
Let $M=M^+\oplus M^-$ be the irreducible representation of $\operatorname{Cliff}(E,q)$. Clifford multiplication by $s$ defines a 2-periodic complex $\underline M\udot=\dots\rt{s\cdot}\underline M^+\rt{s\cdot}\underline M^-\rt{s\cdot}\dots
$ with cohomology supported on $Z(s)=\{0\}\subset\C^n$. Then\vspace{-1mm}
\begin{fleqn}\beq{OH7}\tag{{\bf OH7}}
\textbf{$K$-theoretic}\qquad \sqrt e\;(\underline E,s)\,=\,\cH^+(\underline M\udot)\,-\,\cH^-(\underline M\udot)\,\in\,K\(\mathrm{Coh}\;\{0\}\)\,=\,\Z.
\eeq\end{fleqn}
If $(\underline E,q,s)$ is equivariant with respect to a $T=\C^*$ action on $\C^n$ with $(\C^n)^T=\{0\}$,\vspace{-1mm}
\begin{fleqn}\beq{OH8}\tag{{\bf OH8}}
\textbf{Torus localisation}\qquad \sqrt e\;(\underline E,s\,)\,=\,
\sqrt e^T\!\(\underline E|\_0\)\big/e^T\(N_{0/\C^n}\).
\eeq\end{fleqn}
\vspace{-5mm}

\section{Refinement when $n=2$}
When $n=2$ these quantities can be refined to a pair of topological integers $(d_1,d_2)$, most easily defined by \eqref{RH2} below, whose difference is $\sqrt e\;(\underline E,s)=d_1-d_2$. In the notation and numbering of Section \ref{OHI}, these can be rewritten as follows.

In \eqref{ZZ} we describe a homotopy equivalence from the space of nonzero isotropic vectors in $\C^4$ to $S(E_\R)\times\PP(\Lambda_-)\cong S^3\times S^2$, with $\pi_3=\Z\oplus\Z\ni(s_1,s_2)=\big[s|_{S^3}\big]$.
\beq{RH1}\tag{{\bf RH1}}
(d_1,d_2)\=(s_1+s_2,s_2).\vspace{-2mm}
\eeq
\beq{RH2}\tag{{\bf RH2}}
\big[\PP\(s|_{S^3}\)\big]\=(d_1,d_2)\,\in\,\pi_3(Q)\,=\,\pi_3\(\PP(\Lambda_+)\times\PP(\Lambda_-)\)\=\Z\oplus\Z.
\eeq
If $s$ is holomorphic, write $s=(f\sigma,\,g\;\sigma\ip\,\omega)\in\Gamma\(\underline\Lambda\oplus\underline\Lambda^*)$ as explained in (\ref{sigtau}, \ref{ip}).
\beq{RH3}\tag{{\bf RH3}}
d_1\,=\,\operatorname{length}Z(\sigma)\,\ge\,0, \qquad
d_2\,=\,\operatorname{length}Z(f,g)\,\ge\,0.
\eeq
We prove\footnote{In \eqref{OH4} we \emph{assumed} that there exists such an isotropic deformation to simple zeros; we can currently \emph{prove} one exists only for $n=2$.} there exists a 
holomorphic, isotropic deformation $s_t$ of $s_0=s$ with simple zeros near $0\in\C^2$ for $t\ne0$. Then $d_1,d_2$ count the zeros $\big\{p\in s_t^{-1}(0)\big\}$ of $s_t$ by sign,
\beq{RH4}\tag{{\bf RH4}}
d_1\,=\,\#\big\{p\ \colon(-1)^{|\Lambda_p|}=+1\big\}, \qquad
d_2\,=\,\#\big\{p\ \colon(-1)^{|\Lambda_p|}=-1\big\}.
\eeq
For $s$ homogeneous $s|_{S^3}$ descends to $\underline{s}\colon\PP^1\to Q$. Setting $d_\pm:=\int_{\PP^1}\underline s^*\big[\PP(\Lambda_\pm)\big]$,\vspace{-1mm}
\beq{RH5}\tag{{\bf RH5}}
d_1\=d_-^2\;, \qquad
d_2\=d_+^2\;.
\eeq
Again for $s$ homogeneous we have $d_1=d_-^2$ and $d_2=d_+^2$ where, in $H^2(Q,\Z)$,
\beq{RH6}\tag{{\bf RH6}}
[\PP(C_{Z(s)/\C^2})]\=(d_++d_-)\Big(d_-\big[\PP(\Lambda_+)\big]+d_+\big[\PP(\Lambda_-)\big]\Big).
\eeq
In terms of the 2-periodic Clifford complex $\underline M\udot$,
\beq{RH7}\tag{{\bf RH7}}
d_1\,=\,\operatorname{length}\,\cH^+(\underline M\udot)\,\ge\,0, \qquad
d_2\,=\,\operatorname{length}\,\cH^-(\underline M\udot)\,\ge\,0.
\eeq
If $(\underline E,s)=\(\!\Hom(\underline M^+,\underline M^-),\,F\otimes v\)$ is equivariant with respect to a $T=\C^*$ action on $\C^2$ with $(\C^2)^T=\{0\}=Z(F)=Z(v)$,
\beq{RH8}\tag{{\bf RH8}}
d_1\=e^T(M^-)\big/e^T(N_{0/\C^2}),\qquad d_2\=e^T\((M^+)^*\)\big/e^T(N_{0/\C^2}).
\eeq

\vspace{-2mm}
\setcounter{tocdepth}{1}
\tableofcontents

\section{Running example}
Throughout we will illustrate the results with the following generalisation of \eqref{eg}. Fix integers $d>0$ and $0\le i,j\le d$. We work over the base $\C^2$ with coordinates $x,y$; these will have weights $1,-1$ respectively when we want to use a $\C^*$ action on $\C^2$. We set $E=\C^4$ with coordinates $X,Y,Z,W$, quadratic form $XY+ZW$ and complex orientation (in the sense of \cite[Definition 2.1]{OT1}) $o=\partial_X\wedge\partial_Y\wedge\partial_Z\wedge\partial_W$. We consider the following isotropic section of $\underline E$,
\beq{run}
(x^d,\,y^d,\,x^iy^j,\,-x^{d-i}y^{d-j})\ \text{ of }\ \underline E\,=\,\underline\C^4\,=\,\underline\t^{\;d}\oplus\underline\t^{-d}\oplus\underline\t^{\;i-j}\oplus\underline\t^{\;j-i}.
\eeq
Here $\t$ denotes the standard weight 1 irreducible representation on $\C^*$, so we have given $E$ a $\C^*$ action (that makes $s$ equivariant) that will only be used in the discussion of \eqref{OH8}. Setting $d=2,\,i=1=j$ gives back \eqref{eg}.

We will repeatedly calculate the invariants for this example, via many different methods. Perhaps the easiest are (\ref{OH4}, \ref{RH4}) by deforming $s$ (non-equivariantly!) through the isotropic sections
$$
\bigg(\mathop{{\textstyle\prod}}_{k=1}^d(x-a_k),\ \ \mathop{{\textstyle\prod}}_{k=1}^d(y-b_k),\ \ \mathop{{\textstyle\prod}}_{k\le i}(x-a_k)\cdot\mathop{{\textstyle\prod}}_{\ell\le j}(y-b_\ell),\ \ -\mathop{{\textstyle\prod}}_{k>i}(x-a_k)\cdot\mathop{{\textstyle\prod}}_{\ell>j}(y-b_\ell)\bigg)
$$
with\vspace{-1mm}
\begin{align*}
i(d-j)&\text{\ positive simple zeros at }\,(a_k,b_\ell)_{k\le i,\,\ell>j},\text{ and}\\
j(d-i)&\text{ negative simple zeros at }(a_k,b_\ell)_{k> i,\,\ell\le j}.
\end{align*}
Here the sign of a simple zero $p\in\C^2$ is the sign \cite[Definition 2.2]{OT1} of the maximal isotropic subspace im$\,ds_t|_p\subset\underline E_p$.
Thus the refined invariants of this example are
\beq{result1}
(d_1,d_2)\=\(i(d-j),\,j(d-i)\)
\eeq
so that
\beq{result2}
\sqrt e\;(\underline E,s)\=d_1-d_2\=d(i-j).
\eeq

\begin{rmk} \label{rmk}
For simplicity we fix $d=2,\,i=1=j$ to return to example \eqref{eg}. Here a length 3 zero locus with ideal $(x^2,y^2,xy)$ is deformed to 2 simple zeros. 
This failure of naive deformation invariance is possible because\,---\,if we ignore the isotropic condition\,---\,we have virtual dimension $-2<0$ (from 4 equations in 2 unknowns).\footnote{For a simpler example consider the two equations $x=0=x-a$ in one variable $x$. This cuts out 1 point for $a=0$ and the empty scheme for $a\ne0$.}

Insisting that the section be isotropic imposes relations between these 4 equations, which is what allows for a theory \cite{OT1} in which this problem can be considered to have virtual dimension zero. But the relations are \emph{nonlinear} so we cannot find a rank 2 subbundle $\underline\Lambda\subset\underline E$ in which the equations lie. (If we could, the equations would form a length two regular sequence, restoring naive deformation invariance so that 3 could not become 2 or $1-1=0$.) Thus the failure \eqref{OH3'} of the equality \eqref{OH3} shows that, even locally around $Z(s)$,
$$
\textit{there does not exist a maximal isotropic subbundle
$\underline\Lambda\subset\underline E$ containing $s$.}
$$

We \emph{can} interpret this zero virtual dimension in a Behrend-Fantechi-like way (using real Kuranishi spaces) if we use \emph{real} equations \cite{BJ}, but lengths of real subschemes are not deformation invariant.
\end{rmk}

\begin{rmk}
When $n=2$ we prove in Section \ref{defsz} that \emph{any} holomorphic  isotropic $s$ can be deformed through other holomorphic isotropic sections to have simple zeros near $0\in\C^2$. For $n>2$ the computation \eqref{ZZ} of $\pi_{2n-1}\(\wt Q\)$ below shows there is no topological obstruction to deforming $s$ to have simple zeros, so it can be done with $C^\infty$ sections. We would like to be able to deduce from an Oka principle that it can be done holomorphically too, but we currently cannot because $\big\{v\in E\colon q(v,v)=0\big\}$ is not an Oka manifold (it is singular at the origin!) and there is currently no good theory of singular Oka spaces.
\end{rmk}

\begin{ex}
The following is the ``smallest'' example where it is unclear whether or not we can deform through holomorphic isotropic sections to one with simple zeros.
Over $\C^3$ consider the section
\beq{eg2}
s \= \(x^2, y^2, z^2, xy, xz, yz\),
\eeq
of $\underline{\C}^6$ which is isotropic with respect to the quadratic form
$$
q \= X_1X_2 + X_1X_3 + X_2X_3 - X_4^2 - X_5^2 - X_6^2.
$$
Using \eqref{OH5} it is easy to calculate $d_+ = 3$ and $d_- = 1$ (or vice versa, if we change the orientation). Hence the orthogonal Hopf index is
$$
\sqrt e\(\underline{\C}^6,s\) \= \pm4 \= \pm\operatorname{length}Z(s).
$$
The scheme $Z(s)$ is not lci, so $s$ cannot factor through any maximal isotropic subbundle of $\underline{\C}^6$.
Thus the converse to \eqref{OH3} fails\,---\,the equality
$$
\sqrt e\(\underline E, s\) \= \pm \operatorname{length} Z(s)
$$
can hold even if $s$ does not factor through a maximal isotropic subbundle. \smallskip

Another numerical invariant of $Z(s)$ (which is not deformation invariant, however) is its Hilbert--Samuel multiplicity or, equivalently, the Segre class $s(Z(s),\C^{n}):=\int_{\PP(C)}H^{n-1}$.\footnote{By \cite[Example 4.2.6 (a)]{fulton}, the class $s(Z(s),\C^n)$ depends only on the scheme $Z(s)$ and not on its embedding in $\C^n$.} 
Here $C:=C_{Z(s)/\C^n}$ and $H$ is the hyperplane class on $Q$, which satisfies $H^{n-1}=[\PP(\Lambda_+)]+[\PP(\Lambda_-)]$. The quadric $Q$ is a homogeneous space for $SO(2n,\C)$, so the intersection product of any two effective $(n-1)$-cycles is nonnegative.
In particular $\int_{\PP(C)}[\PP(\Lambda_\pm)] \ge 0$, and so \eqref{OH6} gives the bound
\beq{inequal}
\big|\sqrt e\;(\underline E,s)\big|\ \le\ s(Z(s),\C^n).
\eeq
By \eqref{prod} below equality holds if and only if $[\PP(C)]\in H^{2n-2}(Q)$ is either a multiple of $[\PP(\Lambda_-)]$ or of $[\PP(\Lambda_+)]$. Of course this is true if $s$ factors through a maximal isotropic subbundle. In our example \eqref{eg2} the inequality $4\le8 = s(Z(s),\C^3)$ is not saturated, again confirming that $s$ does not factor through a maximal isotropic. Finally note that when $Z(s)$ is lci, $s(Z(s),\C^n)=$ length $Z(s)$ so that the bound \eqref{inequal} becomes $|\sqrt e\;(\underline E,s)|\le$ length$\,Z(s)$.
\end{ex}

\begin{rmk}
Throughout we work with $E=\C^{2n}$ with $n\ge2$. The $n=1$ case is trivial because $SO(2,\C)=\C^*$, all $SO(2,\C)$ bundles split as $L\oplus L^*$ with $L$ (respectively $L^*$) a positive (respectively negative) maximal isotropic subbundle, isotropic sections $s$ are either $(\sigma,0),\ \sigma\in\Gamma(L)$ or $(0,\tau),\ \tau\in\Gamma(L^*)$ and $\sqrt e\;(L\oplus L^*,s)$ is either $e(L,\sigma)$ or $-e(L^*,\tau)$. And in the odd rank case $E=\C^{2n+1}$ the relevant square root Euler class is 2-torsion and we know of no nice analogue of the local results \eqref{OH1} to \eqref{OH8}.
\end{rmk}

\section{Homotopy groups}\label{hg}
In this Section we prove \eqref{OH1} to \eqref{OH25}. We take $E=\C^{2n} = \R^{2n} \oplus i\R^{2n}$ with $n\ge2$.
The usual inner product $\langle\,\cdot\,,\,\cdot\,\rangle$ on $\R^{2n}$ complexifies to a complex bilinear form on $\C^{2n}$ whose quadratic form is
\beq{otimesC}
q(a+ib,a+ib)\=|a|^2+2i\langle a,b\rangle-|b|^2.
\eeq
Therefore the isotropic vectors of total norm 1 in $\C^{2n}$ form the space
\beq{wtQ}
\wt Q\ :=\ \big\{(a,b)\ \colon\, |a|^2=\tfrac12=|b|^2,\ \langle a,b\rangle=0\big\}\ \subset\ \C^{2n}.
\eeq
The $\C^*$ action on $\C^{2n}$ restricts to a free $S^1$-action on $\wt Q$,
\beq{acts}
e^{i\theta}\ \colon\ (a,b)\Mapsto(a\cos\theta-b\sin\theta,\ b\cos\theta +a\sin\theta)
\eeq
with quotient the quadric $Q\subset\PP^{2n-1}$ of isotropic lines $\C\langle a+ib\rangle\le\C^{2n}$, so\footnote{Since $Q$ and $\wt Q$ are simply connected their homotopy groups are independent of basepoint.}
$$
\pi_{2n-1}\(\wt Q\)\=\pi_{2n-1}(Q).
$$
The bundle structure
\beq{fib}
\xymatrix@R=18pt@C=12pt{S^{2n-2} \ar[r]& \wt Q \ar[d]_-{p} && (a,b) \ar@{|->}[d] && (a,-Ja) \\
& S^{2n-1} \ar@/^{-2ex}/[u]_{-J} && a && a \ar@{|->}[u]}
\eeq
exhibits $\wt Q$ as the sphere bundle $S\(T_{S^{2n-1}}\)$ of the tangent bundle of $S^{2n-1}\ni a$. The section is produced by choosing any orthogonal oriented\footnote{Orthogonal means $\<Ja,Jb\>=\<a,b\>$ for all $a,b\in\R^{2n}$. Oriented means there is a positively oriented basis of the form $\{e_1, Je_1, \dots, e_n, Je_n\}$. The minus sign in the definition of the section will be explained by \eqref{LJ}.} complex structure $J$ on $\R^{2n}$. It splits the long exact sequence of homotopy groups of the fibration \eqref{fib}, so
\begin{eqnarray} \nonumber
\pi_{2n-1}(Q)\=\pi_{2n-1}\(\wt Q\) &\=& \pi_{2n-1}(S^{2n-1})\oplus\pi_{2n-1}(S^{2n-2}) \\ \label{ZZ}
&\=& \left\{\!\!\begin{array}{lc}\Z\oplus\Z & n=2, \\
\Z\oplus (\Z/2\Z) & n>2.\end{array}\right.
\end{eqnarray}
Now projection of $\pi_{2n-1}(Q)$ to the first $\Z$ factor above is $p_*$, so by \eqref{fib},
\beq{sss}
\big[s|_{S^{2n-1}}\big]\ \in\ \pi_{2n-1}\(\wt Q\)\rt{p_*}\pi_{2n-1}(S^{2n-1})\ \ni\ \big[s_+|_{S^{2n-1}}\big].
\eeq
Here $\pi_{2n-1}(S^{2n-1})$ is really the homotopy class of maps from the base $S^{2n-1}\subset\R^{2n}$ to the sphere in the first summand of fibre $\C^{2n}=\R^{2n}\oplus i\R^{2n}$. Any identification of these two spheres fixes an isomorphism $\pi_{2n-1}(S^{2n-1})\cong\Z$. This isomorphism depends only on the relative orientations of the spheres. 

On the fibres we use the standard orientation $o\_\R=e_1\wedge_\R\dots\wedge_\R e_{2n}$ on $\R^{2n}$, which induces both the complex orientation $o=e_1\wedge\dots\wedge e_{2n}$ on $\C^{2n}$ (as in \cite[Section 2.2]{OT1}) and an orientation on $S^{2n-1}\subset\R^{2n}$ by contracting with the outward pointing normal. We do the same for the base so that identifying the two $S^{2n-1}\subset\R^{2n}$\;s in the standard way fixes the standard isomorphism $\pi_{2n-1}(S^{2n-1})\cong\Z$.

So we may take the right hand side of \eqref{sss} as the definition of $\sqrt e\;(\underline E,s)\in\Z$ for now\footnote{For $s$ holomorphic the definition \cite[Section 3.2]{OT1} of $\sqrt e\;(\underline E,s)$ is reviewed in Section \ref{cosec}. It uses cosection localisation \cite{KL} and a splitting $E=\Lambda\oplus\Lambda^*$ into maximal complex isotropic subspaces in place of the real splitting $E=E_\R\oplus iE_\R$ used here. (The former corresponds to writing the quadratic form as $q=\sum_{i=1}^nX_iY_i$, the latter as $q=\sum_{i=1}^{2n}X_i^2$.) Replacing $o\_\C$ by $-o\_\C$ changes $\sqrt e\;(\underline E,s)$ to $-\sqrt e\;(\underline E,s)$. The two multiplicities are the same $\sqrt e\;(\underline E,s)=\deg s_+$ by \cite{OT2}.}
as in \eqref{OH1}, and then \eqref{sss} shows it equals \eqref{OH2}.


To prove \eqref{OH15} take $\theta=\pi/2$ in \eqref{acts} to see the automorphism $(a,b)\mapsto(-b,a)$ is homotopic to $\id_{\wt Q}$. Composing with $p$ and then $-\id_{S^{2n-1}}\sim\id_{S^{2n-1}}$ shows that $p\sim b$, as required.

\subsection*{Generators from maximal isotropic subspaces}
It is a standard fact that the spaces $\mathfrak{acs}^{\;\pm}$ of oriented (respectively oppositely oriented) orthogonal complex structures on $\R^{2n}$ are isomorphic to the spaces $\operatorname{OGr}^\pm(n,2n)$ of positive (respectively negative) maximal isotropic subspaces of $\C^{2n}$,
\beq{acsogr}
\mathfrak{acs}^{\;\pm}\ \cong\ \operatorname{OGr}^\pm(n,2n).
\eeq
The correspondence maps $J\in\mathfrak{acs}^{\;\pm}$ to its (isotropic!) $+i$-eigenspace,
\begin{eqnarray}\label{LJ}
J &\!\Mapsto\!& \Lambda_J\,:=\,\big\{a-iJa\ \colon\ a\in\R^{2n}\big\}\,\subset\,\C^{2n}, \\ \nonumber
(i)\_{\Lambda}\oplus(-i)_{\overline\Lambda}\,=:\,J_\Lambda\! &\!\Mapsfrom\!& \Lambda.
\end{eqnarray}
Here, in defining the inverse, we split $\C^{2n}\,\cong\,\Lambda\oplus\overline\Lambda$ by using complex conjugation on $\C^{2n}:=\R^{2n}\otimes\C$ and noting that\footnote{\label{fono}\,$v\in\Lambda\cap\overline\Lambda\so q(v,\overline v)=0=q(v,v)\so q(\Re v,\Re v)=0=q(\Im v,\Im v)\so v=0$.} $\Lambda\cap\overline\Lambda=\{0\}$. We then defined $J_\Lambda$ to act as $i$ on $\Lambda$ and $-i$ on $\overline\Lambda$.

The points $(a,-Ja)\in\wt Q$, i.e.~those with $|a|^2=\frac12$, are the points of $\Lambda_J$ of unit norm. Thus the sphere $S(\Lambda_J)$ is the image of the section \eqref{fib} for some $J\in\mathfrak{acs}^{\;+}$. It inherits a natural real orientation from the one on $\Lambda_J\subset\C^{2n}$ induced by $J$ (or equivalently, multiplication by $i$) and its outward pointing normal.

\begin{prop}\label{prp}
The section $S^{2n-1}\to S(\Lambda_J)\subset\wt Q$ of \eqref{fib} preserves orientations if and only if $J\in\mathfrak{acs}^+$.
\end{prop}

\begin{proof}
Let $C(Q)\subset\C^{2n}$ denote the cone on $Q$\,---\,i.e.~the variety of isotropic vectors. The $+i$ eigenspace of $J$ \eqref{LJ} is the image of the map
\beq{map}
\R^{2n}\,\rt\sim\,\Lambda_J\Into C(Q)\ \subset\ \C^{2n}, \qquad a\Mapsto a-iJa,
\eeq
intertwining the action of $J$ on $\R^{2n}$ with the action of $i$ on $\C^{2n}\supset C(Q)$. Since complex structures determine orientations, we see \eqref{map} preserves orientations if and only if $J$ is compatible with the orientation on $\R^{2n}$.

Intersecting with spheres on both sides (oriented using the outward normal) we find that the section $S^{2n-1}\rt\sim S(\Lambda_J)\Into S\(C(Q)\)=\wt Q$ of \eqref{fib} preserves orientations if and only if $J$ is compatible with the orientation on $\R^{2n}$.
\end{proof}

For this reason we let
\beq{lpm}
\lambda_+,\,\lambda_-\ \in\ \pi_{2n-1}\(\wt Q\)
\eeq
denote the homotopy class
of the section $S(\Lambda_J)$ for $J\in\mathfrak{acs}^+$ and \emph{minus} the homotopy class of the section $S(\Lambda_J)$ for $J\in\mathfrak{acs}^-$ respectively. Therefore via the isomorphism $\pi_{2n-1}\(\wt Q\)\cong\Z\oplus\Z(/2\Z)$ of \eqref{ZZ} we have
$$
\lambda_+\=(1,0), \qquad \lambda_-\=(-1,k)
$$
for some $k\in\Z(/2\Z)$ which we will show is 1 in Corollary \ref{liver} in Appendix \ref{app}. This proves \eqref{OH25}.

\begin{rmk}\label{mod2} That $k=1$ means that if $s$ can be perturbed through isotropic sections $s_t$ to have simple zeros close to the origin\,---\,$P$ of them positive and $N$ of them negative\,---\,then $s$ determines not only $P-N=\sqrt e\;(\underline E,s)$ but also $P\pmod 2$ and $N\pmod 2$. (And when $n=2$, so that $k\in\Z$ instead of $\Z/2\Z$, it determines $P$ and $N$ completely by \eqref{RH2}.)
\end{rmk}

\section{Homology groups}
In this Section we handle \eqref{OH3} to \eqref{OH6} by calculating $H_{2n-1}\(\wt Q\)$ and the Hurewicz map $\pi_{2n-1}\(\wt Q\)\to H_{2n-1}\(\wt Q\)$. (See \eqref{hur3} for a stronger result.)

\begin{prop} The Hurewicz map fits into the commutative diagram
\beq{Hur}
\xymatrix@R=16pt{
\pi_{2n-1}\(\wt Q\) \ar@{->>}[r]\ar@{=}[d]_{\eqref{ZZ}}& H_{2n-1}\(\wt Q\) \ar@{=}[d] \\
\Z\oplus\Z(/2\Z) \ar[r]^-{(1,0)}& \Z.\!\!}
\eeq
\end{prop}

\begin{proof}
It is a classical fact that
$$
H_{2n-2}\(Q^{2n-2},\Z\)\=\Z\cdot\!\big[\PP(\Lambda_+)\big]\,\oplus\,\Z\cdot\!\big[\PP(\Lambda_-)\big]
$$
with intersection product
\beq{prod}
\big[\PP(\Lambda_+)\big]\cdot\big[\PP(\Lambda_-)\big]\,=\,\left\{\!\!\begin{array}{ll} 1&n\text{ even,} \\ 0&n\text{ odd,} \end{array}\right. \quad \big[\PP(\Lambda_+)\big]^2\,=\,\big[\PP(\Lambda_-)\big]^2\,=\,\left\{\!\!\begin{array}{ll} 0&n\text{ even,} \\ 1&n\text{ odd.} \end{array}\right.\!\!
\eeq
We combine this with the Leray spectral sequence of the fibration $S^1\to\wt Q\to Q$ \eqref{wtQ} to compute $H_{2n-1}\(\wt Q\)$. There is only room for one differential, giving
\beq{LES}
H_0(S^1)\otimes H_{2n}(Q)\rt\partial H_1(S^1)\otimes H_{2n-2}(Q)\To H_{2n-1}\(\wt Q\)\To0.
\eeq
Identifying $H_i(S^1)\cong\Z,\ i=0,1$, the differential $\partial\colon H_{2n}(Q)\to H_{2n-2}(Q)$ is cap product with the Euler class of the circle bundle $\wt Q\to Q$, i.e.~with $c_1(\cO_Q(-1))=-H$, where $H\in H^2(Q)$ is the hyperplane class. Now $H_{2n}(Q)=\Z[H^{n-2}]$, where $[\ \cdot\ ]$ denotes Poincar\'e dual, so $\im\partial$ is generated by
$$
[H^{n-1}]\=\big[\PP(\Lambda_+)\big]+\big[\PP(\Lambda_-)\big]\ \in\ 
\Z\big[\PP(\Lambda_+)\big]\oplus\Z\big[\PP(\Lambda_-)\big]\=H_{2n-2}(Q).
$$
By construction these classes $\big[\PP(\Lambda_\pm)\big]$ represent, in $H_{2n-1}\(\wt Q\)$, the fundamental classes of the circle bundles $S(\cO(-1))$ over $\PP(\Lambda_\pm)$. Since these are the spheres $S(\Lambda_\pm)\subset\wt Q$ from Proposition \ref{prp}, they are in the image of the Hurewicz map $\pi_{2n-1}\(\wt Q\)\to H_{2n-1}\(\wt Q\)$ and
\eqref{LES} now gives
\beq{Hurt}
H_{2n-1}\(\wt Q\)\ \cong\ \frac{\Z\cdot\!\big[S(\Lambda_+)\big]\oplus\Z\cdot\!\big[S(\Lambda_-)\big]}{\Langle\big[S(\Lambda_+)\big]+\big[S(\Lambda_-)\big]\Rangle}\ \xrightarrow[\sim]{(a,b)\,\mapsto\,a-b}\ \Z\cdot\big[S(\Lambda_+)\big].
\eeq
This map gives the right hand vertical arrow of \eqref{Hur} and proves the Proposition.
\end{proof}

\subsection*{Proof of \eqref{OH5}}
When $s$ is homogeneous $s|_{S^{2n-1}}$ factors
\beq{facs}
\xymatrix@R=15pt{
S^{2n-1}\ar[r]^-s\ar[d]& \wt Q \ar[d] \\
\PP^{n-1} \ar[r]_-{\underline s}& Q.\!}
\eeq
Now $\wt Q=S\(\cO_Q(-1)\)$ is the unit circle bundle of $\cO_Q(-1)$, whose pull back $\underline s^*\,\cO_Q(-1)$ to $\PP^{n-1}$ is some $\cO_{\PP^{n-1}}(-d)$. Integrating $c_1^{n-1}$ gives
\beq{homd}
d^{n-1}\,=\,\int_{\PP^{n-1}}\underline s^*\,H^{n-1}\,=\,\int_{\PP^{n-1}}\underline s^*\Big(\big[\PP(\Lambda_+)\big]+\big[\PP(\Lambda_-)\big]\Big)\,=:\,d_++d_-\,.\hspace{-3mm}
\eeq
Therefore $s$ factors as
$$
S\(\cO_{\PP^{n-1}}(-1)\)\,\To\,S\(\cO_{\PP^{n-1}}(-d)\)\=\underline s^*S\(\cO_Q(-1)\)\,\To\,S\(\cO_Q(-1)\).
$$
The first arrow is a bundle map (covering the identity map on $\PP^{n-1}$) and therefore has degree $d$, while the second is induced by $\underline s$. Now in $H_{2n-2}(Q)$, by \eqref{prod},
$$
\underline s\;_*\big[\PP^{n-1}\big]\=\left\{\!\!\begin{array}{ll}d_+\big[\PP(\Lambda_+)\big]+d_-\big[\PP(\Lambda_-)\big] & n \text{ odd}, \\
d_-\big[\PP(\Lambda_+)\big]+d_+\big[\PP(\Lambda_-)\big] & n \text{ even},\end{array}\right.
$$
which we write as $d_\pm[\PP(\Lambda_+)]+d_\mp[\PP(\Lambda_-)]$ where $\pm:=(-1)^{n-1}$. Thus, in $H_{2n-1}\(\wt Q\)$,
$$
s_*\big[S^{2n-1}\big]\=d\,\Big(d_\pm\big[S(\Lambda_+)\big]+d_\mp\big[S(\Lambda_-)\big]\Big).
$$
So by \eqref{Hur} the homotopy classes of $s$ and $d\(d_\pm\big[S(\Lambda_+)\big]+d_\mp\big[S(\Lambda_-)\big]\)$ in $\pi_{2n-1}\(\wt Q\)$ both have the same image in the first $\Z$ factor of \eqref{ZZ}. And by \eqref{Hurt} it is
$$
\sqrt e\;(\underline E,s)\=d\;(d_\pm-d_\mp)\ \stackrel{\eqref{homd}}=\ (-1)^{n-1}(d_++d_-)^{\frac1{n-1}}(d_+-d_-).
$$

\subsection{Another proof of \eqref{OH5}}\label{nother}
We will compactify $(\C^n,\underline E,s)$ to $(\PP^n,\overline E,\overline s)$ in such a way that the quadratic form extends to $\overline E$ and the isotropic section $\overline s$ is nowhere zero on the boundary $\PP^n\take\C^n$. Thus its only zero is the origin with multiplicity $\sqrt e\;(E,s)\in\Z$, which we can then compute as
\beq{ans}
\sqrt e\;(\underline E,s)\=\int_{\PP^n}\sqrt e\;\(\;\overline E\;\).
\eeq

Consider $\C^n\subset\PP(\C^n\oplus\C)=\PP^n$ and let $z\in\Gamma\(\cO_{\PP^n}(1)\)$ be the section vanishing on the infinity divisor $\PP(\C^n\oplus\{0\})=\PP^{n-1}\subset\PP^n$. Let $d$ be the homogeneous degree of $s\in H^0(\underline E)$. 
Give $E\oplus\C\oplus\C$ its obvious quadratic form $q\oplus(XY)$ and let $\underline E\oplus\cO\oplus\cO$ denote the corresponding trivial orthogonal bundle over $\PP^n$. Considering $(s,1,0)$ to be a rational isotropic section (regular over $\C^n\subset\PP^n$) it becomes regular when multiplied by $z^d$. Thus we get an isotropic subbundle 
\beq{line}
\xymatrix{
\cO(-d)\ \ \ar@{^(->}[rr]^-{z^d(s,\,1,\,0)}&& \ \underline E\oplus\cO\oplus\cO\ \text{ over }\ \PP^n}
\eeq
which is orthogonal to the section $(0,-1,0)$. Therefore $\overline E:=\cO(-d)^\perp/\cO(-d)$ is an $SO(2n,\C)$ bundle on $\PP^n$ and $(0,-1,0)$ projects to a section $\overline s$.

Its zeros occur where $(s,1,0)$ is a multiple of $(0,-1,0)$, i.e.~only at the origin $[\underline0:1]\in\PP(\C^n\oplus\C)$. And over $\C^n=\{z\ne0\}$ we have the isometry
$$
\underline E\ \rt{\ e\Mapsto(e,\,0,\,-q(e,s))\ }\ \cO(-d)^\perp/\cO(-d)\=\overline E
$$
mapping $s$ to $\(s,0,0\)\equiv\(s,0,0\)-(s,1,0)=(0,-1,0)=\overline s$. This proves \eqref{ans}.\medskip

To compute the integral \eqref{ans} we pull back a universal calculation of Edidin-Graham via the classifying map
$$
f\colon\PP^n\,\To\,\PP(E\oplus\C_1\oplus\C_2)
$$
for the family \eqref{line} of lines in $E\oplus\C^2$. Since $z^d(s,1,0)$ is isotropic this has image in the quadric $Q^{2n}$ of isotropic lines, over which we have the isotropic subbundle $\cO(-1)\into\underline E\oplus\underline\C^2$. Since $f^*\cO(-1)\cong\cO(-d)$,
\beq{univ}
f^*\(\cO(-1)^\perp/\cO(-1)\)\=\cO(-d)^\perp/\cO(-d)\=\overline E.
\eeq
By \cite[Theorem 4]{EG} we have, in $H^{2n}(Q^{2n})$,
$$
\sqrt e\left(\frac{\cO(-1)^\perp}{\cO(-1)}\right)\=(-1)^{n-1}\Big(\big[\PP(\Lambda_+\oplus\C_2)\big]-\big[\PP(\Lambda_-\oplus\C_2)\big]\Big),
$$
where $\Lambda_\pm$ (respectively $\Lambda_\pm\oplus\C_2$) are maximal isotropic subspaces of $E$ (respectively $E\oplus\C^2$) of the appropriate sign. Now intersect both sides with the hyperplane section $\PP(E\oplus\C_1)$. Since this intersects the $\PP(\Lambda_\pm\oplus\C_2)$ transversally in $\PP(\Lambda_\pm)$,
\beq{univ2}
\sqrt e\!\left.\left(\frac{\cO(-1)^\perp}{\cO(-1)}\right)\right|_{Q^{2n-1}}=\ (-1)^{n-1}\;\iota_*\Big(\big[\PP(\Lambda_+)\big]-\big[\PP(\Lambda_-)\big]\Big)\,\in\,H^{2n}\(Q^{2n-1},\Z\),\hspace{-2mm}
\eeq
where $\iota\colon Q^{2n-2}\into Q^{2n-1}$ is the inclusion\footnote{Note that $Q^{2n-1}$ is singular at the point $[\underline 0:1:0]$ (only), disjoint from $Q^{2n-2}$ and $\PP(\Lambda_\pm)$ whose Poincar\'e duals are therefore defined. The singularity is crucial; it ensures the classes $\iota_*[\PP(\Lambda_+)]$, $\iota_*[\PP(\Lambda_-)]$ of \eqref{univ2} are different ($H^{2n}$ of a smooth $Q^{2n-1}$ only has rank 1).} $\PP(E)\cap Q^{2n}\into\PP(E\oplus\C_1)\cap Q^{2n}$. Now $f$ factors through $Q^{2n-1}$ and pulls back the divisor $Q^{2n-2}$ to $(z^d=0)=d\;\PP^{n-1}\subset\PP^n$. Therefore
pulling \eqref{univ2} back by $f$ using \eqref{univ} gives
$$
\int_{\PP^n}\sqrt e\(\;\overline E\;\)\=(-1)^{n-1}\int_{d[\PP^{n-1}]}\(f|\_{\PP^{n-1}}\)^*\Big(\big[\PP(\Lambda_+)\big]-\big[\PP(\Lambda_-)\big]\Big).
$$
Since on $(z=0)$ we have $f|\_{\PP^{n-1}}=\underline s\colon\PP^{n-1}\to Q^{2n-2}$, this gives
$$
\sqrt e\;(\underline E,s)\,=\,(-1)^{n-1}d\!\int_{\PP^{n-1}}\underline s^*\Big(
\big[\PP(\Lambda_+)\big]-\big[\PP(\Lambda_-)\big]\Big)\,=\,(-1)^{n-1}d(d_+-d_-).
$$
But $d=(d_++d_-)^{1/(n-1)}$ \eqref{homd}, so this reproves \eqref{OH5}. \medskip

To prove \eqref{OH3}, \eqref{OH4} and \eqref{OH6} we degenerate $\Gamma_{\!s}\subset\underline E$ to $\lim_{t\to\infty}\Gamma_{\!ts}$, the isotropic cone $C:=C_{Z(s)/\C^n}\subset\underline E$ supported over $Z(s)=\{0\}\subset\C^n$. This turns $\sqrt e\;(\underline E,s)$ into the class $\surd\;0^{\,!}_E\;[C]$ of \cite[Definition 3.3]{OT1}.

If the cycle $C\subset\underline E|_0=E$ factors (set-theoretically) through a maximal isotropic subspace $\Lambda\subset E$ then \cite[Lemma 3.5]{OT1} gives
$$
\sqrt0^{\,!}_E\;[C]\=(-1)^{|\Lambda|}0^{\,!}_\Lambda\;[C].
$$
This, plus \eqref{H3} applied to $\Lambda$, gives \eqref{OH3}. Combining \eqref{OH3} with the deformation invariance of $\sqrt e\;(\underline E,s)$ also gives \eqref{OH4}.\medskip

So finally we turn to \eqref{OH6}. We use \cite[Lemma 3.4]{OT1} to write $\sqrt e\;(\underline E,s)$ as
$$
\sqrt0^{\,!}_E\;[C]\=\int_{\overline C}\sqrt e\(\cO(-1)^\perp/\cO(-1)\),
$$
where $\overline C\subset\overline E=\PP(E\oplus\C)$ is the projective completion of $C\subset E$. By \eqref{univ2} this is
$$
(-1)^{n-1}\int_{\iota^*[\overline C]}\big[\PP(\Lambda_+)\big]-\big[\PP(\Lambda_-)\big]
$$
where $\iota\colon Q^{2n-2}\into Q^{2n-1}$ is the inclusion $\PP(E)\cap Q^{2n}\into\PP(E\oplus\C)\cap Q^{2n}$ and $\big[\PP(\Lambda_\pm)\big]$ denotes the Poincar\'e dual of these cycles inside $Q^{2n-2}$. Since $\iota^*[\overline C]=\big[\PP(C)\big]$ in $\PP(E)$, this proves \eqref{OH6}.

\section{Cosection localisation}\label{cosec}
To evaluate \eqref{OH3'} we use the definition of $\sqrt e\;(E,s)$ in \cite[Definition 3.3]{OT1}. This uses the orientation $o$ to fix the sign ambiguity in the Edidin-Graham class \cite{EG}, then cosection localisation \cite{KL} to localise to the zeros of the isotropic section $s$. In turn \cite[Definition 3.3]{OT1} uses \cite[Section 3.2]{OT1}, which  uses an auxiliary choice of maximal isotropic subspace $\Lambda\subset E$ (the final result is independent of $\Lambda$, by \cite[Equation 34]{OT1}). In our example \eqref{eg} we may take $\Lambda\subset E$ to be the first and fourth summands $\underline\C^2\subset\underline\C^4$ with quotient $\Lambda^*$ being the second and third.

That is, we work over $Y=\C^2_{x,y}$ with $\underline \C^4=\Lambda\oplus\Lambda^*$ where $\Lambda=\underline\C^2$. In this splitting the quadratic form is given by the obvious pairing and
$$
s\=(s_1,s_2)\ \text{ where }\ s_1\=(x^2,-xy)\ \text{ and }\ s_2\=(y^2,xy).
$$
Note $\langle s_1,s_2\rangle=0$, so $s$ is isotropic. As in \cite[Definition 3.3]{OT1} we replace the graph $\Gamma_{\!s}$ by the cone
\beq{cone1}
C\,:=\,C_{Z(s)/Y}\,=\,\lim_{t\to\infty}\Gamma_{\!ts}\,\subset\,\Lambda\oplus\Lambda^*\,=\,\underline\C^4,
\eeq
where $\lim$ denotes the flat limit in the Hilbert scheme of subschemes of the total space of $\underline\C^4$.
Letting $X,Y,Z,W$ be the obvious linear coordinates on the fibres of $\underline\C^4$\,---\,so $X,W$ on $\Lambda$ and $Y,Z$ on $\Lambda^*$\,---\,we find the ideal of $C$ is
\beq{Cdeal}
\(Z+W,\ XY-Z^2,\ yX-xZ,\ xY-yZ,\ x^2,\,xy,\,y^2\).
\eeq
Following \cite[Section 3.2]{OT1} we now replace this by
$$
C_{C\;\cap\;\Lambda/C}\=\lim_{t\to\infty}t\cdot C,
$$
where $t$ acts on $\Lambda\oplus\Lambda^*$ by $(\id,\,t\cdot\id)$. This amounts to replacing $Y,Z$ by $Y/t,Z/t$ and taking $t\to\infty$, and means the ideal of $C_{C\;\cap\;\Lambda/C}$ is
\beq{idea}
\(W,\ XY,\ yX,\ xY-yZ,\ x^2,\,xy,\,y^2\)\ \text{ with radical }\ (W,XY,x,y).
\eeq
We claim that as a cycle the former is just the latter with multiplicity 2.
When $X\ne0$ the ideal \eqref{idea} becomes $(W,Y,y,x^2)$ with multiplicity $2=\operatorname{length}\,\C[x]/(x^2)$. Over $Y\ne0$ the ideal becomes $(W,X,x-yZ/Y,y^2)$ with multiplicity $2=\operatorname{length}\,\C[y]/(y^2)$. We can ignore the 1-dimensional locus $X=0=Y$ (where $W=0=x^2=xy=y^2$) when considering the 2-dimensional cycle class of $C_{C\;\cap\;\Lambda/C}$, which is therefore $2\times2$ copies of $\C^2$,
\beq{ccone}
\big[C_{C\;\cap\;\Lambda/C}\big]\=2\(\C^2_{X,Z}\ +\ \C^2_{Y,Z}\).
\eeq
The notation gives the coordinates on the $\C^2$\;s with the other variables set to zero\,---\,so the first summand is $x=0=y=Y=W$ with the second $x=0=y=X=W$. In particular both lie in the fibre $\Lambda\_0\oplus\Lambda_0^*$  of $\Lambda\oplus\Lambda^*$ over the origin $x=0=y$. \smallskip

The recipe of \cite[Section 3.2]{OT1} is to now consider this space $ \Lambda_0\oplus\Lambda^*_0=\C^4$ as a $\Lambda_0^*$-bundle over $\Lambda_0=\C^2_{X,W}$, then to consider the tautological section $(X,W)$ of $\Lambda$ as a cosection $\underline\Lambda_0^*\to\cO_{\C^2_{X,W}}$. Then we use the Kiem-Li construction \cite[Equation 2.1]{KL} to cosection localise the intersection of $\big[C_{C\;\cap\;\Lambda/C}\big]$ \eqref{ccone} and $0_{\underline\Lambda_0^*}=(Y=0=Z)$ to the zero locus $0\in\C^2_{X,Z}$ of the cosection.

The definition \cite[Equation 2.1]{KL} treats the two summands of \eqref{ccone} differently. The second is already supported over $0\in\C^2_{X,W}$ so is just intersected with $0_{\underline\Lambda_0^*}=(Y=0=Z)$ to give $2\times 1=2$.

The first component is supported over the $X$-axis $\C_X=(W=0)$, where it is equal to the line subbundle $\underline\C_Z\subset\underline\C^2_{Y,Z}=\underline\Lambda_0^*|_{\C_X}$. The cosection restricts to $(X,0)$ and vanishes on a \emph{divisor} $(X=0)\subset\C_X$. This means we can omit the blowing up used in \cite[Equation 2.1]{KL} and simply take (2 times) \emph{minus} the vanishing locus of $(X,0)$, i.e.~$-2$ times the origin $0\in\C^2_{X,W}$.

In total we get the origin with multiplicity $2-2=0$, recovering \mbox{\eqref{OH3'}} and conclusively demonstrating that the many ways of calculating $\sqrt e\;(E,s)$ in this paper are much simpler to work with than the original definition via cosection localisation.


\section{$K$-theoretic} Here we explain more about \eqref{OH7}. The Clifford algebra Cliff$\;(E,q)$ is a $(\Z/2\Z)$-graded algebra with a unique irreducible graded module $M=M^+\oplus M^-$.
If $s$ is an isotropic section of $\underline E$ then its action on $M$ induces a 2-periodic complex \cite{PV}
\beq{Mdot}
\underline M\udot\=\dots\rt s\underline M^+\rt s\underline M^-\rt s\dots
\eeq
since $q(s,s)=0$. If we pick a splitting $E=\Lambda\oplus\Lambda^*$ into maximal isotropic subspaces, with $\Lambda\subset E$ positive and quadratic form given by the pairing between $\Lambda$ and $\Lambda^*$, then
\beq{M}
M\=M^+\oplus M^-\,\text{ is isomorphic to }\, \bigwedge\nolimits^{\!\bullet}\Lambda^*\=\bigwedge\nolimits^{\!\mathrm{even}}\Lambda^*\oplus\bigwedge\nolimits^{\!\mathrm{odd}}\Lambda^*
\eeq
with Clifford action of $(e,f)\in\Lambda\oplus\Lambda^*$ being $f\wedge(\,\cdot\,)+e\,\ip\,(\,\cdot\,)$. So writing $s=(s_1,s_2)\in\Gamma(\Lambda\oplus\Lambda^*)$ with $\langle s_1,s_2\rangle=0$,
\beq{mdot}
\underline M\udot\=\dots\rt{s_2\wedge\ +s_1\ip}\bigwedge\nolimits^{\!\mathrm{even}}\underline\Lambda^*\rt{s_2\wedge\ +s_1\ip}\bigwedge\nolimits^{\!\mathrm{odd}}\underline\Lambda^*\rt{s_2\wedge\ +s_1\ip}\dots
\eeq
Up to tensoring by a line bundle $(\det\underline\Lambda\otimes K_{\C^n})^{1/2}$\,---\,which for us is trivial\,---\,the $K$-theory class of $\bigwedge^{\!\mathrm{even}}\underline\Lambda^*-\bigwedge^{\!\mathrm{odd}}\underline\Lambda^*$
is precisely the $K$-theoretic square root Euler class $\sqrt\mathfrak{e}\;(\underline E)$ defined in \cite[Section 5.1]{OT1}. And its localisation $\sqrt\mathfrak{e}\;(\underline E,s)\in K(Z(s))$ to $Z(s)$ defined in \cite[Section 5.2]{OT1} is shown in \cite{OS} to coincide with
$$
\cH^+(\underline M\udot)\,-\,\cH^-(\underline M\udot)\ \in\ K\(Z(s)\)\=K\(\mathrm{Coh}\;\{0\}\)\=\Z.
$$
Taking localised Chern character we recover the localised cohomological class by \cite[Theorem 6.1]{OT1}, so
$$
\sqrt e\;(\underline E,s)\=\cH^+(\underline M\udot)\,-\,\cH^-(\underline M\udot)\,\in\,K\(\mathrm{Coh}\;\{0\}\)\,=\,\Z,
$$
which is \eqref{OH7}.\medskip

We now calculate $\cH^\pm(\underline M\udot)$ in our running example \eqref{run}. We let $\Lambda$ be the first and fourth summands of $E$ so that $s_1=(x^d,\,-x^{d-i}y^{d-j})$ and $s_2=(y^d,\,x^iy^j)$. Thus
the complex $\underline M\udot$ \eqref{mdot} becomes
$$
\xymatrix@C=30pt{
\dots \ar[rrr]^-{\,\,\left(\begin{smallmatrix}x^d&-x^{d-i}y^{d-j} \\ -x^iy^j&y^d\end{smallmatrix}\right)\,\,}&&& \underline\C^2 \ar[rrr]^-{\,\,\left(\begin{smallmatrix}y^d&x^{d-i}y^{d-j} \\ x^iy^j&x^d\end{smallmatrix}\right)\,\,}&&& \underline\C^2 \ar[r]&\dots}
$$
The kernel of the second matrix is the kernel of the row vector $(y^j\ x^{d-i})$ (since both its rows are nonzero multiples of this one).
But this is the image of the column vector $(x^{d-i}\ -\!y^j)^t$. Since the first (respectively second) column of the first matrix is $x^i$ (respectively $-y^{d-j}$) times by this column vector, we see that
$$
\cH^+(\underline M\udot)\ \cong\ \C[x,y]/(x^i,y^{d-j}),\quad\text{ of length }i(d-j).
$$
Similarly
$$
\cH^-(\underline M\udot)\ \cong\ \C[x,y]/(y^j,x^{d-i}),\quad\text{ of length }j(d-i).
$$
Subtracting the two not only recovers $\sqrt e\;(E,s)=d(i-j)$, but also its refinement \eqref{RH7} as the difference of $i(d-j)$ and $j(d-i)$. 

\section{Torus localisation}
Suppose $T=\C^*$ acts on $\C^n$ with fixed locus $(\C^n)^T=\{0\}$. Suppose it also acts on $(E,q,o)$, fixing the quadratic form and the orientation, and suppose that the induced action on $\underline E$ preserves the section $s$.

Then as in \cite[Section 7]{OT1} both $\sqrt e\;(E)$ and $\sqrt e\;(E,s)$ lift naturally to equivariant classes with localisation formulae expressing them as classes pushed forward from the fixed loci $(\C^n)^T=\{0\}$ and $Z(s)^T$ respectively, with the former class being the pushforward of the latter. Since in our case $(\C^n)^T=\{0\}=Z(s)^T$ this means they are the same class $\sqrt e^T\!\(\underline E|_0\)\big/e^T\(N_{0/\C^n}\)$. This gives \eqref{OH8}.

This gives an easy computation of our running example \eqref{run}\,---\,the section
$$
s\,=\,(x^d,\,y^d,\,x^iy^j,\,-x^{d-i}y^{d-j})\ \text{ of }\ \underline E\,=\,\underline\t^{\;d}\oplus\underline\t^{-d}\oplus\underline\t^{\;i-j}\oplus\underline\t^{\;j-i} \text{ over }\C^2\,=\,\t^{-1}\times\t.
$$
Choosing the positive maximal isotropic subspace $\Lambda:=\t^d\oplus\t^{\;j-i}\subset E$ gives
$$
\frac{\sqrt e^T\!\(\underline E|_0\)}{e^T\(N_{0/\C^2}\)}\=
\frac{e^T(\Lambda)}{e^T(T_0\C^2)}\=\frac{e(\t^{\;d}\oplus\t^{\;j-i})}{e(\t^{-1}\oplus\t)}\=\frac{(dt)(j-i)\;t}{(-t)\;t}\=d(i-j),
$$
agreeing with \eqref{result2} and \eqref{OH3'} when $d=2,\,i=1=j$. Notice the simplicity of this calculation, which did not even use the section $s$ (just the fact that $Z(s)$ is supported at $0\in\C^2$).\medskip

\section{$n=2$ refinements}\label{refine}
When $n=2$ special things happen\footnote{Also $SO(4,\C)$ is a double cover of $SO(3,\C)\times SO(3,\C)$, corresponding to the splitting of $\bigwedge^2E=\bigwedge^+\oplus\bigwedge^-$ into eigenspaces of the complex Hodge star made from $q$ and $o$. Furthermore, if $E=\Lambda\oplus\Lambda^*$ then $(\sigma,\tau)\in\Gamma(\underline\Lambda\oplus\underline\Lambda^*)$ is isotropic $\langle\sigma,\tau\rangle=0$ if and only if $\sigma\wedge\tau\in\bigwedge^+$. But we have not seen a way to exploit this because the zero locus of $\sigma\wedge\tau$ is bigger than that of $(\sigma,\tau)$ in general.} because $E_\R\cong\R^4\cong\HH$ and $Q\cong\PP^1\times\PP^1$. This will allow us to refine the integer $\sqrt e\;(\underline E,s)$ to a pair of integers $(d_1,d_2)$ with $\sqrt e\;(\underline E,s)=d_1-d_2$. We will give a number of equivalent expressions \eqref{RH1} to \eqref{RH7} for this refinement.

\subsection{Special isomorphisms} Fix an isomorphism $\R^4\cong\HH$ with the quaternions and use the induced metric and orientation. Recall the spaces $\mathfrak{acs}^{\;\pm}$ \eqref{acsogr} of complex structures on $\R^4$ compatible with its metric and orientation (respectively the opposite orientation). For $n=2$ we will denote them by $S^2_\pm$ since they are isomorphic to the space of unit imaginary quaternions,
$$
\mathfrak{acs}^{\;\pm}\ =:\ S^2_\pm\=S(\im\HH),
$$
with the complex structure being given by left\footnote{In our running example \eqref{run} we identify $\R^4$ with $\HH$ by $(X,Y,Z,W)\mapsto X+iY+jZ+kW$. Then \emph{left} multiplication by imaginary quaternions is compatible with our orientation $o=\partial_X\wedge\partial_Y\wedge\partial_Z\wedge\partial_W$; for instance $i\colon\partial_X\mapsto\partial_Y$ and $i\colon\partial_Z\mapsto\partial_W$ because $i\cdot j=k$.} (respectively right) multiplication by the quaternion. This trivialises the bundle \eqref{fib}, giving an isomorphism
\beq{Qt2}
S^3\times S^2_+\ \rt\sim\ \wt Q, \quad (a,J)\Mapsto(a,-Ja),\ \text{ where }\ J\in S(\im\HH).
\eeq
As explained in \eqref{LJ} the $S^2_\pm$ are isomorphic to the spaces $\operatorname{OGr}^{\;\pm}(2,4)$ of positive (respectively negative) maximal isotropic subspaces of the complexification $\R^4\otimes\C$,
$$
S(\im\HH)\ \cong\ S^2_\pm\ \cong\ \operatorname{OGr}^{\;\pm}(2,4)
$$
by mapping $J\in S^2_\pm$ to its $+i$-eigenspace in OGr$\;^\pm(2,4)$,
\begin{eqnarray*}
J &\Mapsto& \Lambda_J\,:=\,\big\{a-iJa\ \colon\ a\in\R^4\big\}\,\subset\,\C^4, \\
(i)\_{\Lambda}\oplus(-i)_{\overline\Lambda}\,=:\,J_\Lambda\! &\Mapsfrom& \Lambda
\end{eqnarray*}
just as in \eqref{LJ}. To see these $S^2$s as $\PP^1$s with a holomorphic structure we intersect with a \emph{fixed} positive (respectively negative) maximal isotropic subspace $\Lambda_\pm\subset\C^4$,
\beq{switch}
S^2_\pm\ \cong\ \operatorname{OGr}^{\;\pm}(2,4)\ \cong\ \PP(\Lambda_\mp), \qquad \Lambda\,\Mapsto\,\Lambda\cap\Lambda_\mp\,.
\eeq
Notice the switch of signs! The reasons these maps are isomorphisms is that any isotropic line $\<v\>\in\C^4$ is contained in a \emph{unique positive maximal isotropic subspace} $\Lambda_v^+\subset\C^4$: it is the set of vectors in $\<v\>^\perp$ which project to the unique positive isotropic line in the $SO(2,\C)$ space $\<v\>^\perp/\<v\>$. Similarly $\<v\>$ is contained in a unique maximal negative isotropic subspace $\Lambda_v^-\subset\C^4$.

Therefore $\<v\>=\Lambda_v^+\cap\Lambda_v^-$ and, conversely, any pair of a positive and a negative maximal isotropic subspace intersect in a line $\langle v\rangle$. Thus $Q$ is just the product of the spaces \eqref{switch},
\beq{arr}
\begin{array}{ccccccc}
Q \!\!&\cong&\!\! \operatorname{OGr}^{\;+}(2,4)\times\operatorname{OGr}^{\;-}(2,4) \!\!&\cong&\!\! \PP(\Lambda_-)\times\PP(\Lambda_+) \!\!&\cong&\!\! 
S^2_+\times S^2_- \hspace{-3mm}\vspace{1mm}\\
\[v\] &\hspace{-5mm}\Mapsto\hspace{-5mm}& \(\Lambda_v^+,\Lambda_v^-\) &\!\!\!\Mapsto\!\!\!&\!\! \(\Lambda_v^+\cap\Lambda_-\;,\, \Lambda_v^-\cap\Lambda_+\) \!\!&\!\!\!\Mapsto\!\!\!& (J_+,J_-)\hspace{-3mm}
\end{array}
\eeq
The last term is the unique $(J_+,J_-)\in S^2_+\times S^2_-$ with $+i$-eigenspace containing $\<v\>\subset\C^4$ (equivalently, $J_\pm\Re v=-\Im v$ in $\R^4$ or $J_\pm=(i)\_{\Lambda_v^\pm}\oplus(-i)_{\overline\Lambda_v^\pm}$ as in \eqref{LJ}).\smallskip

\subsection{Even more special}\label{ems}
Two final descriptions of the maps $Q\to\PP(\Lambda_\pm)$ will be useful later in Section \ref{holl}. We describe them invariantly but the description in local coordinates is much simpler\,---\,for which most readers can safely skip to \eqref{Qq} in the next section.

Since $n=2=\dim\Lambda^+$ a further special thing happens: $\Lambda_+$ is naturally skew-isomorphic to its dual, up to scale. That is, picking a trivialisation $\omega\in\bigwedge^2\Lambda_+^*$ \,---\,which is unique up to scale\,---\,we have $\ip\,\omega\,\colon\Lambda_+\rt\sim\Lambda_+^*$.

Then, given $(\sigma,\tau)\in\Lambda_+\oplus\Lambda_+^*$, the identity $\langle \sigma,\tau\rangle\,\omega\equiv\tau\wedge(\sigma\ip\,\omega)$ shows the isotropic condition $\langle \sigma,\tau\rangle=0$ is equivalent to $\tau$ and $\sigma\ip\,\omega$ being proportional. So writing\footnote{At least for $\sigma\ne0$. Otherwise we have the equivalent expression $(\sigma,\,\tau)=\(\lambda^{-1}\tau\ip\,\omega^{-1},\,\tau\)$.}
\beq{st}
(\sigma,\,\tau)\=\(\sigma,\,\lambda\;\sigma\ip\,\omega\), \quad\lambda\in\PP^1,
\eeq
we will find that
\begin{itemize}
\item fixing $\sigma$ and varying $\lambda\in\PP^1$ gives the elements of a $\PP(\Lambda_-)$, while
\item fixing $\lambda\in\PP^1$ and varying $\sigma\in\Lambda_+$ gives the elements of a $\PP(\Lambda_+)$.
\end{itemize}
Therefore using \eqref{arr} we will show the map to $\PP(\Lambda_+)$ takes $(\sigma,\tau)$ to $[\sigma]$ while the map to $\PP(\Lambda_-)$ takes it to the ratio $\lambda=\tau/(\sigma\ip\,\omega)$.

We use the splitting $\C^4=\Lambda_+\oplus\overline\Lambda_+\cong\Lambda_+\oplus\Lambda_+^*$ of \eqref{LJ} and Footnote \ref{fono}.

\begin{lem}\label{lmm}
The map $Q\to\PP(\Lambda_+)$ takes $(\sigma,\tau)\mapsto[\sigma]\in\PP(\Lambda_+)$ if $\sigma\ne0\ne\tau$.
\end{lem}

\begin{proof}
Denoting $v=(\sigma,\tau)$ we claim that
\beq{lamv}
\Lambda^-_v\=\Langle(\sigma,0),\,(0,\tau)\Rangle.
\eeq
Indeed the isotropic condition $q(v,v)=\tau(\sigma)=0$ implies the right hand side is isotropic. It is maximal because $\sigma\ne0\ne\tau$ and its intersection $\Langle(\sigma,0)\Rangle$ with $\Lambda^+$ is one dimensional because $\sigma\ne0$, so it must be negative. But it also contains $(\sigma,\tau)=v$, and $\Lambda_v^-$ is by definition the unique such subspace. Thus by \eqref{arr} the map $Q\to\PP(\Lambda_+)$ takes $(\sigma,\tau)$ to $\Lambda_v^-\cap\Lambda_+=\<\sigma\>$.
\end{proof}

Next we show the map to a fixed $\PP(\Lambda_-)$ is effectively given by the ratio $\tau/(\sigma\ip\,\omega)\in\PP^1$. Since $\Lambda_-\cap\Lambda_+$ is 1-dimensional we may pick a basis vector $\sigma'$.
 
\begin{lem}\label{lmm2}
The projection $Q\to\PP(\Lambda_-)$ takes $(\sigma,\tau)\mapsto\big[(\sigma',\lambda\sigma'\ip\,\omega)\big]$ where $\lambda=\tau/(\sigma\ip\,\omega)\in\PP^1$.
\end{lem}

\begin{proof}
Denoting $v=(\sigma,\tau)$ and $\lambda=\tau/(\sigma\ip\,\omega)\in\C\cup\{\infty\}$ as above, we claim
\beq{lamv+}
\Lambda^+_v\=\big\{(\sigma'',\,\lambda\sigma''\ip\,\omega)\,\colon\,\sigma''\in\Lambda_+\big\} \quad\text{if }\lambda\ne\infty
\eeq
and $\Lambda^+_v=\Lambda_+^*$ if $\lambda=\infty$. Indeed the right hand side of \eqref{lamv+} is isotropic, maximal, positive (since it tends to $\Lambda_+$ as $\lambda\to0$) and contains $v$.

Intersecting with $\Lambda_-=\Langle(\sigma',0),\,(0,\sigma'\ip\,\omega)\Rangle$ gives
\begin{equation}
\Langle(\sigma',\lambda\sigma'\ip\,\omega)\Rangle\,\in\,\PP(\Lambda_-),
\end{equation}
as claimed. In the basis $\{\sigma',\,\sigma'\ip\,\omega\}$ for $\Lambda_-$ this is just $\lambda\in\PP^1$.
\end{proof}

\subsection*{Local coordinates}
Choose linear coordinates $X,W$ on $\Lambda_+$ and dual coordinates $Y,Z$ on $\Lambda_+^*$ so that the quadratic form on $\C^4=\Lambda\oplus\Lambda^*$ is $XY+ZW$. Then $\Lambda_+$ is a positive maximal isotropic subspace with respect to the complex orientation $o=\partial_X\wedge\partial_Y\wedge\partial_Z\wedge\partial_W$ by \cite[Definition 2.2]{OT1}.

Thus $Q$ is $\{XY+ZW=0\}\subset\PP^3$ and, by Lemma \ref{lmm}, the projection $Q\to\PP(\Lambda_+)$ is $[X:Y:Z:W]\mapsto[X:W]$.

Choose $\omega=dX\wedge dW$ and write $(\sigma,\tau)=(X,W,Y,Z)=X\partial_X+W\partial_W+Y\partial_Y+Z\partial_Z$. Then $\sigma\ip\,\omega=X\partial_Z-W\partial_Y$ so by Lemma \ref{lmm2} the projection $Q\to\PP(\Lambda_-)$ takes $(X,Y,Z,W)$ to the ratio $\tau/\sigma\ip\,\omega=(Y,Z)/(-W,X)$, i.e.~$[-W:Y]=[X:Z]\in\PP^1$.

In summary we have
\begin{eqnarray} \nonumber
\PP^3\,\supset\,Q\,=\,\{XY+ZW=0\} &\rt\sim& \PP(\Lambda_+)\times\PP(\Lambda_-),\\ \label{Qq}
[X:Y:Z:W] &\Mapsto& \!\Big([X:W]=[-Z:Y]\Big)\times\Big([X:Z]=[-W:Y]\Big) \\ \nonumber
\text{with inverse}\hspace{3cm} \\ [ac:-bd:ad:bc] &\Mapsfrom& [a:b]\times[c:d]. \nonumber
\end{eqnarray}

\subsection*{Example}
Applying \eqref{Qq} to the running example $(x^d,y^d,x^iy^j,-x^{d-i}y^{d-j})$ of \eqref{run} shows the induced map $\C^2\take\{0\}\to\PP(\Lambda_+)\times\PP(\Lambda_-)$ is 
\beq{nonhom}
(x,y)\ \Mapsto\ \Big([x^i:-y^{d-j}],\ [x^{d-i}:y^j]\Big)
\eeq
with homotopy class \eqref{RH2}
\beq{pdog}
(d_1,d_2)\=\(i(d-j),\,(d-i)j\)\ \in\ \Z\oplus\Z\=\pi_3\(\PP(\Lambda_+)\times\PP(\Lambda_-)\)
\eeq
recovering \eqref{result1} and $\sqrt e\;(\underline E,s)=d_1-d_2=d(i-j)$.

\subsection{Homotopy and homology}
We take \eqref{RH2} as the definition of the topological refinement $(d_1,d_2)$ of $\sqrt e\;(\underline E,s)=d_1-d_2$. That is, we write the class of $\PP\(s|_{S^3}\)$ as $(d_1,d_2)$ with respect to $\pi_3(Q)=\pi_3\(\PP(\Lambda_+)\times\PP(\Lambda_-)\)=\Z\oplus\Z$. 

We start by proving \eqref{RH1}. Via \eqref{Qt2} the $S^1$ bundle $\wt Q\to Q$ becomes the projection  $S^3\times S^2_+\to S^2_+\times S^2_-$ taking $(a,b)\in\wt Q$ to the unique complex structures $(J_+,J_-)$ (one oriented, one anti-oriented) which preserve the plane $\langle a,b\rangle\subset\R^4$. That is $(a,b)\mapsto(J_+,J_-)\in S(\im\HH)\times S(\im\HH)$ such that $a=J_+b=bJ_-$\,, so $(J_+,J_-)=(ab^{-1},b^{-1}a)$. Composing with \eqref{Qt2} gives
\beq{qts2}
S^3\times S^2_+\ \cong\ \wt Q\ \To\ S^2_+\times S^2_-\,,\quad (a,J)\Mapsto(J,\,a^{-1}Ja).
\eeq
Consider its restriction to $S^3\times\{J\}\to S^2_+\times S^2_-\;$. To the first factor it is the constant map to $J\in S^2_+$. To the second it is the Hopf map $a\mapsto a^{-1}Ja$. This explains the first column of the matrix below. Similarly on $\{a\}\times S^2_+$ the map to $S^2_+$ is the identity, and to $S^2_-$ is homotopic to the identity (by homotoping $a\in\HH$ to $1\in\HH$). Thus the map on $\pi_3$ induced by \eqref{qts2} is\vspace{-3mm}
\beq{mat}
\xymatrix@C=60pt{
\pi_3(S^3\times S^2_+)\ \stackrel{\eqref{ZZ}}\cong\ \Z\oplus\Z\ 
\ar[r]^-{\left(\begin{smallmatrix}0&1 \\ 1&1\end{smallmatrix}\right)}&
\ \Z\oplus\Z\ \cong\ \pi_3(S^2_+\times S^2_-).}
\eeq
Using the left hand description from \eqref{ZZ}, we write the homotopy class of $s|_{S^3}$ in $\pi_3\(\wt Q\)$\,---\,or equivalently $\PP\(s|_{S^3}\)\in\pi_3(Q)$\,---\,as $(s_1,s_2)$. Recall from \eqref{sss} and \eqref{OH1} that $s_1$ is $\deg s_+=\deg s_-=\sqrt e\;(\underline E,s)$.

Instead writing $\PP\(s|_{S^3}\)\in\pi_3(Q)$ in the basis
\beq{d1d2}
\pi_3(Q)\=\pi_3\(\PP(\Lambda_+)\times\PP(\Lambda_-)\)\ \cong\ \Z\oplus\Z\ \ni\ (d_1,d_2),
\eeq
we find from \eqref{mat} and the sign switch \eqref{switch} that $(d_1,d_2)=(s_1+s_2,s_2)$.

This defines $(d_1,d_2)$ with $d_1-d_2=\sqrt e\;(\underline E,s)$ and proves \eqref{RH1}.\medskip

Now suppose $s$ is homogeneous as in \eqref{RH5}, giving a commutative diagram\vspace{-1mm}
$$
\xymatrix@R=15pt@C=30pt{S^3 \ar[d]_h\ar[r]^{\PP\(s|_{S^3}\)}& Q \ar@{=}[d] \\
\PP^1 \ar[r]^-{\underline s}& \PP(\Lambda_+)\times\PP(\Lambda_-)}
$$
with $h$ the Hopf map. Now $H_2(\,\cdot\,)=\pi_2(\,\cdot\,)$ for both $\PP^1$ and $\PP^1\times\PP^1$ so
by \eqref{prod} the homotopy class of $\underline s$ is $(d_-\;,d_+)\in\pi_2\(\PP(\Lambda_+)\times\PP(\Lambda_-)\)$, where $d_\pm:=\int_{\PP^{n-1}}\underline s^*\big[\PP(\Lambda_\pm)\big]$. And since $(\ \cdot\ )\circ h\,\colon\pi_2(S^2)\to\pi_3(S^2)$ is the squaring map\footnote{Represent $d\in\pi_2(\PP^1)$ by $[x:y]\mapsto[x^d:y^d]$ and lift to $S^3\ni(x,y)\mapsto(x^d,y^d)\in\C^2\take\{0\}\simeq S^3$. This has degree $d^2$, so represents $d^2\in\pi_3(S^3)\xrightarrow[\raisebox{3pt}{\tiny{$\sim$}}]{h_*}\pi_3(S^2)$.}
$\Z\to\Z, \ d\mapsto d^2$, this proves
\eqref{RH5},
$$
\PP\(s|_{S^3}\)\=\underline s\circ h\=(d_-^2\;,\,d_+^2)\ \in\ \Z\oplus\Z\=\pi_3\(\PP(\Lambda_+)\times\PP(\Lambda_-)\).
$$


\subsection*{Example}
Our running example \eqref{run} is homogeneous if and only if $d=i+j$. In that case \eqref{nonhom} gives $\underline s=[x^i:-y^i]\times[x^j:y^j]$ and so $(d_-\;,d_+)=(i,j)$. Therefore $(d_1,d_2)=(i^2,j^2)$, agreeing with \eqref{pdog}.

\subsection{Projectivised normal cone} Here we prove \eqref{RH6} for \emph{homogeneous} $s$ of degree $d$. (Note the isotropic deformation $(x(x-t),y(y-t),xy,-(x-t)(y-t))$ of $(x^2,y^2,xy,-xy)$ shows that \eqref{RH6} is not invariant under general non-homogeneous deformations of $s$.) We actually prove a version of the formula for \emph{any} $n\ge2$.

Letting $d_\pm=\int_{\PP^{n-1}}\underline s^*\big[\PP(\Lambda_\pm)\big]$ as in \eqref{OH5} then $d^{n-1}=d_++d_-$ by \eqref{homd}. Reusing $\pm$ to denote the sign $(-1)^{n-1}$, the generalisation of \eqref{RH6} we will prove is
\beq{RH6n}\tag{{\bf RH6$_n$}}
\big[\PP(C)\big]\=dd_\pm[\PP(\Lambda_+)\big]+dd_\mp[\PP(\Lambda_-)\big]\ \in\ H^{2n-2}(Q,\Z),
\eeq
where $C:=C_{Z(s)/\C^n}$. To prove it we just need to show $\int_{\PP(C)}\big[\PP(\Lambda_+)\big]=dd_+$ and $\int_{\PP(C)}\big[\PP(\Lambda_-)\big]=dd_-\;$, by \eqref{prod}. We already know the difference of these numbers,
$$
\int_{\PP(C)}\big[\PP(\Lambda_+)\big]-\big[\PP(\Lambda_-)\big]\=dd_+-dd_-
$$
by (\ref{OH5}, \ref{OH6}), so we are left with showing their sum
\beq{toprove}
\int_{\PP(C)}\big[\PP(\Lambda_+)\big]+\big[\PP(\Lambda_-)\big]\ =\,\int_{\PP(C)}H^{n-1}\ \text{ equals }\ dd_++dd_-\stackrel{\eqref{homd}}=\,d^n.
\eeq
On the left hand side we have used that $[\PP(\Lambda_+)\big]+[\PP(\Lambda_-)\big]\in H^{2n-2}(Q)$ is the restriction of the $(n-1)$th power $H^{n-1}$ of the hyperplane class $H$ from $\PP(E)\supset Q$. So to prove \eqref{toprove} we may (and do) push forward to $\PP(E)$ instead of working in $Q$.\medskip

As in Section \ref{nother} we compactify $\C^n\subset\PP^{n}$ and let $z\in\Gamma\(\cO_{\PP^n}(1)\)$ be the section vanishing on the infinity divisor $\PP^n\take\C^n=\PP^{n-1}$. The rational section $(s,1)$ of the trivial bundle $\underline E\oplus\cO$ over $\PP^n$ has poles of order $d$ over $z=0$, so the regular map
$$
\xymatrix@C=40pt{
\cO(-d)\ \ \ar@{^(->}[r]^-{z^d(s,\,1)}& \ \underline E\oplus\cO\ \text{ over }\ \PP^n}
$$
defines a line subbundle. It is classified by a map (the projectivisation of $(s,1)$) $\overline s\colon\PP^n\to\PP(E\oplus\C)$ such that $\overline s^*\cO(-1)\cong\cO(-d)$. Therefore
\beq{d2}
\int_{\PP^n}\overline s^*H_2^n\=\int_{\PP^n}(dH_1)^n\=d^n,
\eeq
where $H_1$ and $H_2$ are the hyperplane classes of $\PP^n$ and $\PP(E\oplus\C)$ respectively.

Now we deform the graph $\Gamma_{\!\overline s}\subset\PP^n\times\PP(E\oplus\C)$ through the graphs of $\overline s_t=[s:t]\in\PP(E\oplus\C)$. Taking the flat limit $t\to0$ in the Hilbert scheme of $\C_t\times\PP^n\times\PP(E\oplus\C)$ defines a cycle $C_1+\overline C\subset\PP^n\times\PP(E\oplus\C)$\,---\,described below\,---\,with
\beq{d22}
\int_{C_1}H_2^n+\int_{\overline C}H_2^n\=d^n
\eeq
by \eqref{d2}. Here $H_2$ is the pullback of $H_2$ from $\PP(E\oplus\C)$ to $\PP^n\times\PP(E\oplus\C)$.

Considering $\PP(E\oplus\C)$ to be the projective completion of $E\subset\PP(E\oplus\C)$ we see that over $\C^n\subset\PP^n$ the section $\overline s_t$ is just $s/t$. Therefore one component of the flat $t\to0$ limit is the projective completion $\overline C=\PP\(C\oplus\cO_{Z(s)}\)$ of the cone $C\subset\underline E|_{Z(s)}$.

The other component $C_1$ is given by taking the graph of the section $\lim_{t\to0}\overline s_t=\overline s_0=[s:0]\in\PP(E\oplus\C)$ where this is defined\,---\,i.e.~away from $\{0\}=Z(s)\subset\C^n\subset\PP^n$\,---\,and then taking its closure. Thus it is birational to the base $\PP^n$ (and isomorphic to it away from $Z(s)$). But, over $\C^n\take\{0\}$, the homogeneity of $s$ means $[s:0]$ is just the composition
$$
\C^n\take\{0\}\,\To\ \PP^{n-1}\,\rt{\underline s}\,\PP(E)\ \subset\ \PP(E\oplus\C)
$$
where $\underline s$ is as in \eqref{facs}. This has $(n-1)$-dimensional image in $\PP(E\oplus\C)$ so $\int_{C_1}H_2^n=0$ and \eqref{d22} becomes
$$
d^n\=\int_{\overline C}H_2^n\=\int_{\PP(C)}H^{n-1}.
$$
This proves \eqref{toprove} and so \eqref{RH6n} for all $n\ge2$.

\begin{rmk}
Paolo Aluffi suggested a quicker proof of \eqref{toprove}. Let $\rho\colon\Bl_0\C^n\to\C^n$ be the blow up of $\C^n$ in the origin with exceptional divisor $D\cong\PP^{n-1}$. Then $\rho^*s$ vanishes on $D$ with multiplicity $d$, so $Z(\rho^*s)=dD$ (plus a sum of other irreducible divisors which are empty since $Z(s)$ is supported at the origin). Therefore, by the birational invariance of Segre classes, $s(Z(s),\C^n)=\int_{\PP(C)}H^{n-1}$ is
$$
\rho_*\,s(dD,\Bl_0\C^n)\=\rho_*\(dD - d^2D^2 +\dots + (-1)^{n-1} d^n D^n\)\=d^n.
$$
\end{rmk}

\subsection{$K$-theoretic}\label{holl}
We use $K$-theory to give a holomorphic definition of the refinement $(d_1,d_2)$ which we will prove coincides with \eqref{d1d2}. So we work over $\C^2$ with trivial quadratic bundle $\underline E$ and suppose the isotropic section $s$ is holomorphic. Thus we can form the 2-periodic complex $\underline M\udot$ \eqref{Mdot} and prove \eqref{RH7}.

\begin{prop}
Suppose $s$ is holomorphic with homotopy class \eqref{d1d2} given by $\big[\PP\(s\big|_{S^3}\)\big]=(d_1,d_2)\in\pi_3\(\PP(\Lambda_+)\times\PP(\Lambda_-)\)$. Then
$$
\operatorname{length}\cH^+(\underline M\udot)\=
d_1, \quad \operatorname{length}\cH^-(\underline M\udot)=d_2.
$$
\end{prop}

In particular\,---\,in the special case of $n=2$\,---\,the lengths of these cohomology sheaves is a deformation invariant; not just their difference $\sqrt e\;(\underline E,s)=d_1-d_2$.

\begin{proof}
Write $s=(\sigma,\tau)$ with respect to $E=\Lambda_+\oplus\overline\Lambda_+\cong\Lambda_+\oplus\Lambda_+^*$, as in Footnote \ref{fono}, for some positive maximal isotropic subspace $\Lambda_+$. By perturbing $\Lambda_+\subset E$ if necessary we may assume that neither of $\sigma$ and $\tau$ are identically zero. Let $(f=0)$ and $(g=0)$ be the maximal divisors contained in $Z(\sigma)$ and $Z(\tau)$ respectively, so
\beq{sigtau}
\sigma\=f\sigma_0, \quad \tau\=g\tau_0,
\eeq
where $\sigma_0\in\Gamma(\underline\Lambda_+)$ and $\tau_0\in\Gamma(\underline\Lambda_+^*)$ have 0-dimensional zero loci (possibly empty). Since $Z(s)=Z(\sigma)\cap Z(\tau)$ is supported at the origin, the same is true of $Z(f)\cap Z(g)$; in particular $f$ and $g$ are coprime.

Because $\dim\Lambda_+=2$ a choice of nonzero $\omega\in\bigwedge^2\Lambda_+^*$ trivialises $\bigwedge^2\underline\Lambda_+^*$. The identity $\langle \sigma,\tau\rangle\,\omega\equiv\tau\wedge(\sigma\ip\,\omega)$ shows the isotropic condition $\langle \sigma,\tau\rangle=0$ implies that $\tau$ and $\sigma\ip\,\omega$ are proportional. In fact, scaling $f$ by a unit if necessary, we get
\beq{ip}
\tau_0\=\sigma_0\;\ip\,\omega.
\eeq
We now show that $d_1$ and $\cH^+(\underline M\udot)$ enumerate the zeros of $\sigma_0$ while $d_2$ and $\cH^-(\underline M\udot)$ enumerate the zeros of $(f,g)$. Recall from \eqref{mdot} that $\underline M\udot$ is
$$
\xymatrix@C=32pt{
\dots \ar[r]& \underline\Lambda_+^* \ar[rr]^-{(f\sigma_0\ip\ ,\,g\;\tau_0\wedge\ )}&& \bigwedge^0\underline\Lambda_+^*\oplus\bigwedge^2\underline\Lambda_+^* \ar[rr]^-{(g\;\tau_0\wedge\ ,\,f\sigma_0\ip\ )}&& \underline\Lambda_+^* \ar[r]& \dots}
$$
Using the trivialisation $\omega$ of $\bigwedge^2\underline\Lambda_+^*$ and \eqref{ip} we write this as
$$
\xymatrix@C=32pt{
\dots \ar[r]& \underline\Lambda_+^* \ar[rr]^-{(f\sigma_0\ip\ ,\,g\sigma_0\ip\ )}&& \cO\oplus\cO \ar[rr]^-{(-g\;\sigma_0\ip\ ,\,f\sigma_0\ip\ )}&& \underline\Lambda_+^* \ar[r]& \dots}
$$
Since codim\,$Z(\sigma_0)=2$ the components of $\sigma_0$ form a regular sequence and its Koszul complex is exact. Furthermore $f,g$ are nonzero. Thus the kernel of the first labelled arrow is the image of $\sigma_0\ip\,\colon\bigwedge^2\underline\Lambda_+^*\to\underline\Lambda_+^*$ while the image of the second labelled arrow is the product of this with the ideal $(f,g)$. Thus
\beq{odd}
\cH^-(\underline M\udot)\ \cong\ \cO/(f,g).
\eeq
Similarly the coprimality of $f,g$ implies the exactness of the Koszul complex of $(f,g)$, making the kernel of the second labelled arrow equal to im$\,(f,g)\,\colon\cO\to\cO\oplus\cO$. Inside this copy of $\cO$ is im$\(\sigma_0\ip\,\colon\underline\Lambda_+^*\to\cO\)$; this is the image of the first arrow. Thus
\beq{even}
\cH^+(\underline M\udot)\ \cong\ \cO_{Z(\sigma_0)}.
\eeq
By definition $d_1$ is the homotopy class of the projection of $s|_{S^3}=(\sigma,\tau)|_{S^3}$ via $Q\to\PP(\Lambda_+)$. By Lemma \ref{lmm} this projection is $[\sigma]\in\PP(\Lambda_+)$ at any point of $S^3$ at which $\sigma\ne0\ne\tau$. Since $\sigma=f\sigma_0$ and $\tau=g\sigma_0\ip\,\omega$ this is $[\sigma_0]\in\PP(\Lambda_+)$ at \emph{all} points of $S^3$ by continuity. That is,
$$
d_1\=\big[\PP\(\sigma_0|_{S^3}\)\big]\ \in\ \pi_3(\PP(\Lambda_+))\ \cong\ \Z.
$$
But by \eqref{H3} this is precisely the length of \eqref{even}.

Similarly Lemma \ref{lmm2} shows that $d_2$ is the homotopy class of the map $S^3\to\PP^1$ given by the ratio
$$
(\sigma\ip\,\omega)/\tau\big|_{S^3}\=f/g\big|_{S^3}\ \colon S^3\To\PP^1.
$$
So $d_2$ is the length of \eqref{odd} by \eqref{H3}.
\end{proof}

\subsection{Deforming to simple zeros}\label{defsz} Notice that \eqref{odd}, \eqref{even} give \eqref{RH3}. Furthermore
the descriptions (\ref{sigtau}, \ref{ip}) of $s=(\sigma,\tau)$ also allow us to prove \eqref{RH4} and the fact that\,---\,when $n=2$\,---\,the section $s_0:=s$ can be deformed in a family of holomorphic isotropic sections $s_t$ which have simple zeros close to $0\in\C^2$ for $t\ne0$.

We first deform $\sigma_0$ in a family of sections $\sigma_t\in\Gamma(\underline\Lambda_+)$ which have simple zeros close to $0\in\C^2$ for $t\ne0$. Next we
set $\tau_t:=\sigma_t\ip\,\omega$ so that $\langle\sigma_t,\tau_t\rangle=0$. Then similarly deform $(f,g)$ in a family $(f_t,g_t)$ which has, for $t\ne0$, only simple zeros close to $0\in\C^2$ which are disjoint from those of $\sigma_t$. This gives a family of isotropic sections
$$
s_t\ :=\ (f_t\;\sigma_t,\,g_t\;\tau_t\)\,\in\,\Gamma(\underline E)
$$
with simple zeros for $t\ne0$, specialising to $s=(\sigma,\tau)=(f\sigma_0,g\;\tau_0)$ at $t=0$.

In particular we can calculate the above quantities $d_1=\operatorname{length}\cH^+(\underline M\udot)$ and $d_2=\operatorname{length}\cH^-(\underline M\udot)$
at $t\ne0$ where $s_t$ has simple zeros. There we get $+1$ at each of the zeros of $\sigma_t$ and $-1$ at each of the zeros of $(f_t,g_t)$ as in \eqref{RH4}.

\subsection{Another model}
There is a more attractive model of rank 4 orthogonal vector spaces and their isotropic vectors than we used in Section \ref{ems}. It is
$$
(E,q)\=\(\!\End V,\,\det\!\),\  \text{ where }\,V\,=\,\C^2.
$$
Thus the quadric $Q$ is the space of rank 1 endomorphisms up to scale. Writing them as $F\otimes v\in V^*\otimes V$, we choose the orientation so that the positive (respectively negative) maximal isotropic subspaces are given by fixing $F$ and varying $v$ to give $\langle F\rangle\otimes V$ (respectively fixing $v$ and varying $F$ to give $V^*\otimes\langle v\rangle$). The map to $\PP^1\times\PP^1$ of Lemmas \ref{lmm}, \ref{lmm2} is therefore just $[v]\times[F]$ (or equivalently $(\coker,\,\ker)$).

A slight generalisation is given by 
$$
(E,q)\=\(\!\Hom(M^+,M^-),\,\det\!\)
$$
where $M^\pm$ are 2-dimensional vector spaces with a fixed isomorphism
\beq{deti}
\bigwedge\nolimits^{\!2}M^+\ \cong\ \bigwedge\nolimits^{\!2}M^-.
\eeq
Thus any $\phi\in\Hom(M^+,M^-)$ has a determinant $\bigwedge^2\phi\in\bigwedge^2(M^+)^*\otimes\bigwedge^2M^-\cong\C$ which is canonically a scalar, defining a quadratic form on $\Hom(M^+,M^-)$. The isotropic vectors are rank $\le1$ homomorphisms $F\otimes v$ with $F\in(M^+)^*$ and $v\in M^-$.

\begin{rmk}
When $E=\Lambda\oplus\Lambda^*$ is a sum of positive maximal isotropic subspaces we have already seen an example of this in \eqref{M} with
\begin{multline*}
M^+\,=\,\C\oplus\bigwedge\nolimits^{\!2}\Lambda^*\,\text{ and }\,M^-\,=\,\Lambda^*\,\text{ so that} \\
\Hom(M^+,M^-)\=\Lambda^*\ \oplus\ \bigwedge\nolimits^{\!2}\!\Lambda\otimes\Lambda^*\ \cong\ \Lambda^*\oplus\Lambda\ \cong\ E.
\end{multline*}
In this situation we have seen \eqref{st} that an isotropic vector $(\sigma,\tau)\in\Lambda\oplus\Lambda^*$ can be written as $(\sigma,\,\lambda\sigma\ip\,\omega)$, where $\lambda=\tau/\sigma\ip\,\omega$, so we can take
$$
F\otimes v\ :=\ (\lambda\;,\,\omega^{-1})\otimes(\sigma\ip\,\omega)\ \in\ (M^+)^*\otimes M^-.
$$
Alternatively we can take $F\otimes v=(1,\,\lambda^{-1}\omega^{-1})\otimes\tau$. The first expression works for $\sigma\ne0$ (i.e.~$\lambda\ne\infty$) while the second works for $\tau\ne0$ (i.e.~$\lambda\ne0$).
\end{rmk}

To get a torus localisation formula for the refined index $(d_1,d_2)$ we allow more general spin modules $M^\pm$ but then ask for them and their sections $F,v$ to be equivariant with respect to a $T=\C^*$ action on $\C^2$, for $(\C^2)^T\!,\,Z(F),\,Z(v)$ to be supported at the origin, and for \eqref{deti} to hold \emph{equivariantly}. Then by \eqref{H8} we have
\begin{eqnarray}\nonumber
d_1 &\=& \operatorname{length}Z(v)\=e^T(M^-)\big/e^T\(N_{0/\C^2}\), \\
d_2 &\=& \operatorname{length}Z(F)\=e^T\((M^+)^*\)\big/e^T\(N_{0/\C^2}\), \label{fv}
\end{eqnarray}
proving \eqref{RH8}. Writing $M^+=\t^a\oplus\t^b$ and $M^-=t^i\oplus\t^j$ then \eqref{deti} gives $a+b=i+j$. And $(\t^a)^*\otimes(\t^i\oplus\t^j)$ is a positive maximal isotropic subspace for the quadratic form det (since it consists entirely of homomorphisms of rank $\le1$). Thus by \eqref{OH8},
$$
\sqrt e\;(\underline E,s)\=\frac{e^T\(\t^{i-a}\oplus\t^{j-a}\)}{e^T(N_{0/\C^2})}\=\frac{(i-a)(j-a)\;t^2}{e^T(N_{0/\C^2})}\,.
$$
Since $a+b=i+j$ this is compatible with \eqref{fv}, as it should be,
$$
\sqrt e\;(\underline E,s)\=\frac{(ij-ab)\;t^2}{e^T(N_{0/\C^2})}\=
\frac{e^T(M^-)-e^T((M^+)^*)}{e^T(N_{0/\C^2})}\=d_1-d_2.
$$
Finally we note that in this setting the 2-periodic complex $\underline M^\bullet$ \eqref{Mdot} takes the form
\beq{Mdo}
\dots\To\underline M^+\rt{F\otimes v}\underline M^-\rt{F\otimes v}\underline M^+\To\dots
\eeq
Since $F\otimes v\in (M^+)^*\otimes M^-$ this first arrow is clear. The second is the composition
$$
\underline M^-\rt{\wedge v}\bigwedge\nolimits^{\!2}\underline M^-\xrightarrow[\sim]{\eqref{deti}}\bigwedge\nolimits^{\!2}\underline M^+\rt{F\wedge1}\underline M^+.
$$
This shows \eqref{Mdo} is a complex (because $v\wedge v=0=F\wedge F$) with cohomology
$$
\cH^+(\underline M\udot)\ \cong\ \bigwedge\nolimits^{\!2}\underline M^-\big/v\wedge\underline M^-\ \cong\ \cO/(\im v), \qquad \cH^-(\underline M\udot)\ \cong\ \cO/(\im F),
$$
cf. (\ref{odd}, \ref{even}).

\subsection*{Example} Finally we calculate \eqref{RH8} in our running example \eqref{run}. We let $\C^*$ act on $\C^2$ with weights $(1,-1)$ on $x,y$ and take the section
$$
\mat{y^d}{x^{d-i}y^{d-j}}{x^iy^j}{x^d}\=\(y^j,\,x^{d-i}\)\otimes\(y^{d-j},\,x^i\)
$$
of $\underline E=(\underline M^+)^*\otimes\underline M^-=(\underline\t^{-j}\oplus\underline\t^{d-i})\otimes(\underline\t^{j-d}\oplus\underline\t^i)$.  Then \eqref{fv} gives
\beqa
d_1 &\=& \frac{e^T(M^-)}{e^T(N_{0/\C^2})}\=\frac{i(j-d)\;t^2}{-t^2}\=(d-j)i, \\
d_2 &\=& \frac{e^T((M^+)^*)}{e^T(N_{0/\C^2})}\=\frac{(i-d)j\;t^2}{-t^2}\=(d-i)j.
\eeqa
in agreement with \eqref{result1}.

\section{Applications to virtual cycles}

\subsection{Cosection localised Behrend-Fantechi virtual cycles}
If a projective scheme $M$ has a Behrend-Fantechi virtual cycle\,---\,or if it is a Fulton-MacPherson refined Euler class $e(\underline E,s)$\,---\,and has both virtual and actual dimension zero,\footnote{We can generalise this to asking that $\dim M=\vd(M)\ge0$ by integrating an insertion (of cohomological degree $2\vd(M)$) which has the special property that it cuts $M$ down to dimension zero by imposing an effective algebraic condition which locally, in holomorphic Kuranishi charts, is just intersection with a smooth codimension-$\vd$ subvariety of the local ambient space.\label{insert}} then its length \emph{is} the virtual cycle,
$$
\deg[M]^{\vir}\=\operatorname{length}M\ \ge\ 0.
$$

Even if it has virtual dimension 0 but strictly positive actual dimension, it may still have a cosection which vanishes in dimension zero. In this situation Kiem and Li \cite{KL} define a 0-dimensional localisation of the virtual cycle, and one might ask if its length is the degree of the virtual cycle.

We can interpret 
our main example \eqref{eg} as showing that this is not the case. We will see it gives a 0-dimensional (ordinary, not square rooted) refined Euler class $e\(\underline \Lambda,s_1=(x^2,-xy)\)$, cosection-localised to the zero locus $x^2=0=xy=y^2$ of $s_2=(y^2,xy)$, but which turns out to be 0. The calculation is related to (but not the same as\,---\,in fact it is a deformation of) the cosection localisation calculation of Section \ref{cosec}; see Footnote \ref{6}.\smallskip

We use the notation of Section \ref{cosec}. So we work over $\C^2_{x,y}$ with $\underline \Lambda\,=\,\underline \C^2$ and
$$
\text{section } s_1\,=\,(x^2,-xy)\,\in\,\Gamma(\underline\Lambda) \text{ and cosection } s_2\,=\,(y^2,xy)\,\in\,\Gamma(\underline\Lambda^*).
$$
Note $s_2\circ s_1=0$ implies the cosection condition that the composition
$$
T_Y|_{Z(s_1)}\ \rt{ds_1|_{Z(s_1)}}\ \underline\Lambda|_{Z(s_1)}\ \rt{s_2}\ \cO_{Z(s_1)}
$$
vanishes.

Fulton-MacPherson intersection theory, or Behrend-Fantechi virtual cycle theory, first replaces the graph $\Gamma_{\!s_1}$ by the cone\footnote{Contrast this with \eqref{cone1}, where $s_1$ and $s_2$ were scaled simultaneously. Here only $s_1$ is scaled.\label{6}}
$$
C\,:=\,C_{Z(s_1)/\C^2}\,=\,\lim_{t\to\infty}\Gamma_{\!ts_1}\,\subset\,\underline\Lambda\,\text{ with ideal }\,(x^2, xy, Xy+Wx)
$$
in the linear coordinates $X,W$ on the fibres of $\underline\Lambda=\underline\C^2$ that we used in Section \ref{cosec}.

Over $y\ne0$ the ideal is $(x,X)$, so one irreducible component of the cone is $\C_{0,y}\times\C_{0,W}$ cut out by $x=0=X$.

On $X\ne0$ the ideal is $y=-xW/X,\ x^2=0$. On $W\ne0$ it is $x=-yX/W,\ y^2=0$. So the other component is $\C^2_{X,W}$ supported over the origin $x=0=y$ but with multiplicity 2 everywhere. That is, as cycles,
$$
C\=(\C_{0,y}\times\C_{0,W})+2\;\C^2_{0,0,X,W}\ \subset\ \underline\Lambda\=\C^2_{x,y}\times\C^2_{X,W}.
$$
Intersecting with the 0-section $0_{\underline\Lambda}$ would give a cycle supported over the $y$-axis. But Kiem-Li \cite{KL} give us a way to localise further to the zeros of the cosection
$$
\Gamma(\underline\Lambda^*)\ \ni\ s_2\=(y^2,\,xy)\ \colon\ \underline\Lambda\ \To\ \cO_{\C^2_{x,y}},
$$
i.e.~to the origin $0\in\C^2_{x,y}$. Notice that $s_2$ vanishes on $C\subset\underline\Lambda$. Their recipe \cite[Equation 2.1]{KL} treats the two components of $C$ differently.

For the first we work over the $y$-axis on which the cone $C\subset\underline\Lambda$ is the line subbundle $\C_W\subset\C^2_{X,W}$. The cosection restricts to $(y^2,0)$ which vanishes on a divisor, so we can omit the blowing up used in \cite[Equation 2.1]{KL} and simply take \emph{minus} the vanishing locus of $(y^2,0)$, i.e.~$-2$ times the origin.

The second component is already supported over the origin $0\in\C^2_{x,y}$ so by \cite[Equation 2.1]{KL} we simply intersect it with the 0-section $0_{\underline\Lambda^*}=(X=0=W)$ to get $+2$ times the origin.

In total we get the origin with multiplicity $2-2=0$.

\subsection{DT$^4$ virtual cycles}
Let $M$ be a $(-2)$-shifted symplectic derived projective scheme\,---\,such as a ``DT$^4$ moduli space" of stable sheaves (or complexes of sheaves) on a Calabi-Yau 4-fold when stable\,=\,semistable.

Any point of $M$ has a local neighbourhood with a ``Darboux chart". This local model is the zero locus $Z(s)$ of an isotropic section of an orthogonal bundle $(\underline E,q)$ over a smooth ambient space $A$ \cite{BBJ}. When $M$ has ``orientation data" $(\underline E,q)$ inherits a complex orientation \cite[Proposition 4.2]{OT1}, making it an $SO(r,\C)$ bundle.

Then $M$ carries a virtual cycle $[M]^{\vir}$ \cite{BJ, OT1} of complex dimension $v=\dim_\C A-\frac12r$ which is zero unless $v\in\Z_{\ge0}$ \cite{OT2}. So from now on we fix $r=2n$.

Suppose now that $\dim M=0=v$.\footnote{This condition can be weakened to $\dim M=v\ge0$ just as in Footnote \ref{insert}.} Then it is natural to ask if the invariant  $\int_{[M]^{\vir}}1$ equals the length of $M$ as a 0-dimensional scheme.
Our examples (\ref{eg}, \ref{run}) show that this need not be the case\,---\,that $M$ can be nonempty but with virtual cycle zero. Furthermore if we can perturb the setup so that $M$ deforms to a union of $P$ positive reduced points and $N$ negative reduced points, then not only is $P-N$ constrained to be $\int_{[M]^{\vir}}1$\vspace{-1mm} but both $P$ and $N$ are fixed mod $2$ by Remark \ref{mod2}. This uses the $\Z/2\Z$ factor of $\pi_{2n-1}\(\wt Q\)$ \eqref{ZZ} and its stability under the stabilisation of Corollary \ref{liver} on replacing $\(\C^n,\,\underline\C^{2n},s\)$ by $\(\C^n\times\C_x,\,\underline\C^{2n}\oplus\underline\C^2,\,s\oplus(x,0)\)$.

\subsection*{$(-2)$-shifted cotangent bundles}
There is one important, special situation, however, where the analogue of \eqref{H3} does hold. That is when
$$
M\=T^*[-2]N
$$
for some quasi-smooth derived scheme $N$. This holds, for instance, if $M$ is the moduli space of compactly supported stable sheaves on the local Calabi-Yau 4-fold $X=K_Y$, where $Y$ is a 3-fold whose moduli space of stable sheaves $N$ is quasi-smooth (for example if $Y$ is Fano).

Then \cite[Section 8]{OT1} $M$ admits local \emph{Darboux} charts $(A,\underline E,s)$ with $\underline E=\underline\Lambda\oplus\underline\Lambda^*$ and $s=(s_1,0)$, where $(A,\underline\Lambda,s_1)$ are local \emph{Kuranishi} charts for $N$. In this situation
$$
v\=\vd(N), \quad [M]^{\vir}\=[N]^{\vir}
$$
(where the latter is the usual Behrend-Fantechi virtual cycle of $N$) and $\dim M=0=v$ happens if and only if $\dim N=0=\vd(N)$, in which case, by \eqref{OH3},
\beq{yifan}
\int_{[M]^{\vir}}1\=\mathrm{length}\,M\=\mathrm{length}\,N
\eeq

\section{Morals}

One moral of this note is that counting sheaves on Calabi-Yau 4-folds\,---\,or indeed counting in any algebro-geometric virtual enumerative setting with a cosection\,---\,behaves more like \emph{real} counting than complex counting. Even though it has an entirely holomorphic (indeed algebraic) definition \cite{OT1}, points of the moduli space of the correct dimension can contribute zero or a negative number, rather like counting zeros of the equation $x^2-a=0$ over $\R$ instead of $\C$.

But when the Calabi-Yau 4-fold is a local 3-fold, the $(-2)$-shifted symplectic structure is a $(-2)$-shifted cotangent bundle, or the cosection vanishes, the theory behaves like the counting of zeros over $\C$.

\appendix
\section{Stabilisation}\label{app}

Here we study how the homotopy group $\pi_{2n-1}\(\wt Q\)\cong\pi_{2n-1}(Q)$ of \eqref{ZZ} stabilises as $n\mapsto n+1$. So we record $n$ as a label and do a trivial rescaling of \eqref{wtQ}, defining
$$
\wt Q_n\ :=\ \big\{(a,b)\ \colon\, |a|^2=1=|b|^2,\ \langle a,b\rangle=0\big\}\ \subset\ \R^{2n}\oplus\R^{2n}.
$$
To relate this to $\wt Q_{n+1}$ we will repeatedly use the elementary fact that for any topological subspace $S\subset V$ of a vector space we may write its double suspension as
\beq{suss}
\Sigma^2S\=\Big\{\(\sqrt{1-|X|^2}\cdot v,\,X\)\ \colon\,v\,\in\,S,\,X\in\R^2,\,|X|\le1\Big\}\ \subset\ V\oplus\R^2.
\eeq
Thus we can define an inclusion $\iota \colon \Sigma^2\wt Q_n\into\wt Q_{n+1}$ by
\beq{ota}
\(\sqrt{1-|X|^2}\cdot (a, b),\,X\) \Mapsto \Big(\(\sqrt{1 - |X|^2}\cdot a,\,X\),\ \(\sqrt{1-|X|^2}\cdot b,\,-j_+X\)\Big),
\eeq
where $j_+$ is the oriented orthogonal complex structure on $\R^2$ (and the minus sign is just as in \eqref{LJ}). We first show this stabilises the classes $\lambda_\pm\in\pi_{2n-1}\(\wt Q_n\)\cong\pi_{2n-1}(Q_n)$ of \eqref{lpm}.

\begin{lem}\label{658}
The map $\iota_* \circ \Sigma^2\colon\pi_{2n-1}\(\wt Q_n\)\To\pi_{2n+1}\(\wt Q_{n+1}\)$ takes $\lambda_\pm$ to $\lambda_\pm$.
\end{lem}

\begin{proof}
Fixing an oriented orthogonal complex structure $J_+$ on $\R^{2n}$, recall from \eqref{fib} that $\lambda_+$ is the class of the section
$$
S^{2n-1}\,\ni\,a\Mapsto(a,-J_+a)\,\in\,\wt Q^n\,\subset\,\R^{2n}\oplus\R^{2n}.
$$
Write $S^{2n+1}\cong\Sigma^2S^{2n-1}$ as $\big\{\(\sqrt{1-|X|^2}\cdot a,X\)\big\}$ as in \eqref{suss}. Then \eqref{ota} shows that $\iota_* \circ \Sigma^2(\lambda_+)$ is the map
\beqa
S^{2n+1}\ \cong\ \Sigma^2S^{2n-1} &\To& \wt Q_{n+1}\ \subset\ \R^{2n+2}\oplus\R^{2n+2}, \\
\(\sqrt{1-|X|^2} \cdot a,\,X\) &\Mapsto& 
\Big(\(\sqrt{1-|X|^2} \cdot a,\,X\),\ \(\sqrt{1-|X|^2}(-J_+a), \,-j_+X\)\Big).
\eeqa
Since $(J_+,j_+)$ is an oriented orthogonal complex structure on $\R^{2n}\oplus\R^2$, this is just $\lambda_+\in\pi_{2n+1}\(\wt Q_{n+1}\)$. The case of $\lambda_-$ is similar, using that $(J_-,j_+)$ is an anti-oriented orthogonal complex structure on $\R^{2n}\oplus\R^2$.
\end{proof}

Let $f_n\colon S^{2n-2}\into\wt Q_n$ denote the inclusion of the fibre $S(a^\perp)$ of the fibration $p$ \eqref{fib} over some fixed $a\in S(\R^{2n})$. Let $\epsilon_n$ be the generator of $\pi_{2n-1}(S^{2n-2})\cong\Z(/2\Z)$. Then $f_n\circ\epsilon_n\!\in\!\pi_{2n-1}\(\wt Q_n\)\!\cong\!\Z\oplus\Z(/2\Z)$ represents $(0,1)$ under the isomorphism \eqref{ZZ}.

\begin{lem}\label{659}
$\iota_* \circ \Sigma^2\colon\pi_{2n-1}\(\wt Q_n\)\To\pi_{2n+1}\(\wt Q_{n+1}\)$ maps $f_n\circ\epsilon_n\!\Mapsto\!f_{n+1}\circ\epsilon_{n+1}$.
\end{lem}

\begin{proof}
Using the standard fact that $\Sigma^2\epsilon_n=\epsilon_{n+1}$,
$$
\iota_*\circ\Sigma^2\(f_n\circ\epsilon_n\)\=
\iota_*\circ\Sigma^2\(f_n\)\circ\epsilon_{n+1},
$$
so it is sufficient to prove that $\iota_*\circ\Sigma^2\(f_n\)$ is homotopic to $f_{n+1}$. By \eqref{suss} and \eqref{ota},
\beqa
\iota_*\circ\Sigma^2\(f_n\)\,\colon\,\Sigma^2S(a^\perp) &\To& \wt Q_{n+1}\ \subset\ \R^{2n+2}\oplus\R^{2n+2}, \\
\(\sqrt{1-|X|^2} \cdot b,\,X\) &\Mapsto& 
\Big(\(\sqrt{1-|X|^2} \cdot a,\,X\),\ \(\sqrt{1-|X|^2}\cdot b, \,-j_+X\)\Big),
\eeqa
where $a$ is fixed, $b$ ranges over $S(a^\perp)$ and $X$ ranges over $\R^2$ with $|X|\le1$. Now homotope from $t=1$ to $t=0$ along the maps
$$
\Big(\(\sqrt{1-t|X|^2} \cdot a,\,tX\),\ \(\sqrt{1-|X|^2}\cdot b, \,-j_+X\)\Big),
$$
noting they also lie in $\wt Q_{n+1}$. The result is the map
$$
\Big(\(a,\,0\),\ \(\sqrt{1-|X|^2}\cdot b, \,-j_+X\)\Big)\ \sim\ \Big(\(a,\,0\),\ \(\sqrt{1-|X|^2}\cdot b,\,X\)\Big)
$$
by homotoping $-j_+$ to $\id_{\R^2}$ along isometries of $\R^2$. But this is $f_{n+1}$, the fibre of \eqref{fib} over $(a,0)\in S(\R^{2n}\oplus\R^2)$.
\end{proof}

\begin{cor}\label{liver}
For all $n$ we have $\lambda_-=(-1,1)\in\pi_{2n-1}\(\wt Q_n\)\cong\Z\oplus\Z(/2\Z)$.
\end{cor}

\begin{proof}
By Lemmas \ref{658} and \ref{659} the map $\iota_*\circ\Sigma^2\colon\pi_{2n-1}\(\wt Q_n\)\To\pi_{2n+1}\(\wt Q_{n+1}\)$ is
$$
\mat1001\colon\,\Z\oplus\Z(/2\Z)\ \To\ \Z\oplus (\Z/2\Z)
$$
with respect to the basis $\lambda_+=(1,0),\ f_n\circ\epsilon_n=(0,1)$ of \eqref{ZZ}. It also maps $\lambda_-\mapsto\lambda_-$ and so preserves the relation $\lambda_-=(-1,1)$; this holds for $n=2$ by \eqref{mat} so it holds for all $n$ by induction.
\end{proof}

Combining Corollary \ref{liver} with \eqref{ZZ} and \eqref{Hur} shows that for $n>2$ the Hurewicz map sits in the following commutative diagram of short exact sequences
\beq{hur3}
\xymatrix@R=14pt@C=20pt{
&&&&&& \Z/2\Z \ar[d] \\
0 \ar[r]& \Z \ar[d]_2\ar[rr]^-{(2,2)}&& \Z\oplus\Z \ar@{=}[d]\ar[rrr]^-{\lambda_+\oplus\lambda_-}&&& \pi_{2n-1}\(\wt Q_n\) \ar[d]\ar[r]& 0 \\
0 \ar[r]& \Z \ar[d]\ar[rr]^-{(1,1)}&& \Z\oplus\Z \ar[rrr]^-{[S(\Lambda_+)]\oplus [S(\Lambda_-)]}&&& H_{2n-1}\(\wt Q_n\) \ar[r]& 0.\!\! \\
& \Z/2\Z}
\eeq

\subsection*{Acknowledgements} We are grateful to Nick Kuhn for suggesting \eqref{OH7}, and to him, Paolo Aluffi, Franc Forstneri\v c, Achim Krause, Richard L\"ark\"ang, Samuel Mu\~noz-Ech\'aniz, Sasha Polishchuk and Yifan Zhao for useful conversations. MK acknowledges support from NWO Grant VI.Vidi.192.012 and ERC Consolidator Grant FourSurf 101087365. JO was supported by the New Faculty Startup Fund of Seoul National University and a National Research Foundation of Korea grant funded by the Korean government (MSIT)(RS-2024-00339364).
JVR was supported by Research Council of Norway grant no.~302277.
RT was supported by a Royal Society research professorship.

\bigskip

\noindent \begin{tabular}{lcl}
{\tt{m.kool1@uu.nl}} && {\tt{jeongseok@snu.ac.kr}} \smallskip \\
Department of Mathematics && Department of Mathematical \\
Utrecht University && Sciences and RIM \\
PO Box 80010 && Seoul National University \\
3508 TA Utrecht && Seoul 08826 \\
The Netherlands &$\qquad$& Korea \\ \\
{\tt{jorgeren@math.uio.no}} && {\tt{richard.thomas@imperial.ac.uk}}\smallskip \\
Department of Mathematics && Department of Mathematics \\
University of Oslo && Imperial College London \\
PO Box 1053 Blindern && London SW7 2AZ \\
0316 Oslo, Norway && United Kingdom
\end{tabular}

\end{document}